\documentclass[final,3p,times]{elsarticle}       

\usepackage{lineno} 
\modulolinenumbers[5]

\makeatletter
\def\ps@pprintTitle{%
  \let\@oddhead\@empty
  \let\@evenhead\@empty
  \let\@oddfoot\@empty
  \let\@evenfoot\@oddfoot
}
\makeatother
\usepackage[table,xcdraw,svgnames]{xcolor}
\usepackage[colorlinks]{hyperref}
\AtBeginDocument{%
  \hypersetup{
    citecolor=magenta,
    linkcolor=blue,
    urlcolor=Blue}}
\usepackage[heading=true,scheme=plain]{ctex} 
\usepackage{mathrsfs}
\usepackage{amsfonts}   
\usepackage{amssymb}   
\usepackage{lmodern}   
\usepackage{bbm}  

\usepackage{autobreak}       
\allowdisplaybreaks[4]   

 \usepackage[capitalise]{cleveref}

  \usepackage{extarrows}  
\usepackage{tikz}   
\usetikzlibrary{arrows.meta}
\usetikzlibrary{snakes}
\usepackage{pgfplots}
\usetikzlibrary{patterns}
\usepackage{tkz-euclide}
\usetikzlibrary{calc}
\usetikzlibrary{intersections}
\usepackage{subfigure}       
\numberwithin{figure}{section}
\usepackage[justification=centering]{caption}  
 \usepackage[shortlabels]{enumitem}  
\numberwithin{equation}{section} 
 \allowdisplaybreaks[4]   
 \usepackage[amsmath,thmmarks]{ntheorem} 
\newtheorem{definition}{Definition} [section]            
\newtheorem{theorem}{Theorem}[section]            
\newtheorem{lemma}{Lemma} [section]

\newdefinition{corollary}{Corollary}[section]
\newtheorem{remark}{Remark}

  \theoremsymbol{\hfill $\Box$} 
 \newtheorem*{proof}{Proof}
 \theoremstyle{empty}
 \newtheorem{refproof}{Proof}  

\def\dchi{\scalebox{1.2}{$\chi$}}

\def\dint{\displaystyle\int}

\DeclareMathOperator*{\bmo}{BMO}

\DeclareMathOperator*{\essinf}{ess\, inf}
\DeclareMathOperator*{\esssup}{ess\, sup}
\DeclareMathOperator*{\lip}{Lip}

\newcommand{\mathd}{\mathrm{d}}

\biboptions{numbers,sort&compress}   

\begin{document}

\begin{frontmatter}

\title{{\bfseries  Characterization of Lipschitz Spaces via Commutators of Fractional Maximal Function on the $p$-Adic Variable Exponent Lebesgue  Spaces }}

\author[mymainaddress]{Jianglong Wu\corref{mycorrespondingauthor}}
\cortext[mycorrespondingauthor]{Corresponding author}
\ead{jl-wu@163.com}

\author[mymainaddress]{Yunpeng Chang}
\ead{yunpeng_chang2023@163.com}


\address[mymainaddress]{Department of Mathematics, Mudanjiang Normal University, Mudanjiang  157011, China}

\begin{abstract}

In this paper, the main aim is to give some characterizations of the boundedness of the maximal or nonlinear commutator of the $p$-adic fractional maximal operator $ \mathcal{M}_{\alpha}^{p}$  with the symbols belong to the $p$-adic Lipschitz spaces in the context of the  $p$-adic version of    variable   Lebesgue spaces,
by which some new characterizations of the  Lipschitz spaces and nonnegative Lipschitz
functions   are obtained in the $p$-adic field context.
Meanwhile, Some equivalent relations between the  $p$-adic Lipschitz norm and the  $p$-adic
  variable   Lebesgue norm are also given.

\end{abstract}

\begin{keyword}
$p$-adic field \sep Lipschitz function \sep  fractional maximal function \sep  variable exponent Lebesgue space

\MSC[2020]   42B35  
              \sep 11E95  
                  \sep 26A16 
                   \sep 26A33  	
                  \sep 47G10 
\end{keyword}

\end{frontmatter}


\section{Introduction and main results}
\label{sec:introduction}

During the last several decades, 
the $p$-adic analysis has attracted extensive attention due to 
its important applications in mathematical physics, science and technology, for instance, 
$p$-adic harmonic analysis, $p$-adic pseudo-differential equations, $p$-adic wavelet theory, etc (see \cite{albeverio2006harmonic,shelkovich2009padic,khrennikov2010non,torresblanca2023some}).
 Moreover, 
the theory of variable exponent function spaces  has been intensely investigated in the past twenty years since some elementary properties were established by Kov{\'a}{\v{c}}ik  and R{\'a}kosn{\'\i}k in \cite{kovavcik1991spaces}.

It is worthwhile to note that the  fractional maximal operator plays an important role in real and harmonic analysis and applications, such as potential theory and partial differential equations (PDEs),  since it  is intimately related to the Riesz potential operator, which  is a powerful tool in the study of  the smooth function spaces
 (see \cite{folland1982hardy,bonfiglioli2007stratified,carneiro2017derivative}).

On the other hand,  there are two major reasons  why the study of the  commutators  has got widespread attention. The first one is that the boundedness of commutators can produce some characterizations of function spaces \cite{janson1978mean,paluszynski1995characterization}. The other one is that the theory of commutators is intimately related to the regularity properties of the solutions of certain PDEs  \cite{chiarenza1993w2,difazio1993interior,ragusa2004cauchy,bramanti1995commutators}.

Let $T$ be the classical singular integral operator. The Coifman-Rochberg-Weiss type commutator $[b, T]$ generated by $T$ and a
suitable function $b$ is defined by
\begin{align}  \label{equ:commutator-1}
 [b,T]f      & = bT(f)-T(bf).
\end{align}
A well-known result shows that  $[b,T]$ is bounded on $L^{s}(\mathbb{R}^{n})$  for $1<s<\infty$ if and only if $b\in \bmo(\mathbb{R}^{n})$ (the space of bounded mean oscillation functions). The sufficiency was provided  by \citet{coifman1976factorization} and the necessity was obtained by  \citet{janson1978mean}.
Furthermore, \citet{janson1978mean} also established some characterizations of the Lipschitz space $\Lambda_{\beta}(\mathbb{R}^{n})$ via commutator \labelcref{equ:commutator-1} and  proved that  $[b,T]$ is bounded from $L^{s}(\mathbb{R}^{n})$ to $L^{q}(\mathbb{R}^{n})$   for $1<s<n/\beta$ and $1/s-1/q=\beta/n$ with $0<\beta<1$ if and only if  $b\in \Lambda_{\beta}(\mathbb{R}^{n})$  (see also  \citet{paluszynski1995characterization}).

Denote by $\mathbb{N}$, $\mathbb{Z}$, $\mathbb{Q}$ and $\mathbb{R}$ the sets of positive integers, integers, rational numbers and real numbers, respectively.
For  $\gamma\in \mathbb{Z}$ and a prime number $p$,
let $\mathbb{Q}_{p}^{n}$ be a vector space over the $p$-adic field $\mathbb{Q}_{p}$, $B_{\gamma} (x)$ denote a $p$-adic ball with center $x \in \mathbb{Q}_{p}^{n}$ and radius $p^{\gamma}$ ( for the notations and notions, see  \cref{sec:preliminary} below).

Let  $0 \le \alpha<n$, for  a locally integrable function  $f$, the $p$-adic  fractional maximal function of $f$ is defined by
\begin{align*}
  \mathcal{M}_{\alpha}^{p}(f)(x) = \sup_{\gamma\in \mathbb{Z}  \atop x\in \mathbb{Q}_{p}^{n}} \dfrac{1}{|B_{\gamma} (x)|_{h}^{1-\alpha/n}} \dint_{B_{\gamma} (x)} |f(y)| \mathd y,
\end{align*}
where the supremum is taken over all $p$-adic balls $B_{\gamma} (x)\subset \mathbb{Q}_{p}^{n}$ and $|E|_{h}$ represents the Haar measure of a measurable set $E\subset\mathbb{Q}_{p}^{n}$.
When $\alpha=0$, we simply write  $\mathcal{M}^{p}$ instead of $\mathcal{M}_{0}^{p}$, which is the $p$-adic Hardy-Littlewood maximal function defined as
\begin{align*}   
  \mathcal{M}^{p}(f)(x) = \sup_{\gamma\in \mathbb{Z}  \atop x\in \mathbb{Q}_{p}^{n}} \dfrac{1}{|B_{\gamma} (x)|_{h}} \dint_{B_{\gamma} (x)} |f(y)| \mathd y.
\end{align*}

The reader can refer to Stein \cite{stein1993harmonic} for the definition on the Euclidean case.

Similar to  \labelcref{equ:commutator-1},  we can define two different kinds of commutator of the fractional maximal function as follows.
\begin{definition} \label{def.commutator-frac-max}
 Let  $0 \le \alpha<n$ and $b$ be  a locally integrable function on $\mathbb{Q}_{p}^{n}$.
\begin{enumerate}[label=(\arabic*),itemindent=1em]
\item The maximal commutator of $ \mathcal{M}_{\alpha}^{p}$ with $b$ is given by
\begin{align*}
\mathcal{M}_{\alpha,b}^{p} (f)(x) &= \sup_{\gamma\in \mathbb{Z}  \atop x\in \mathbb{Q}_{p}^{n}} \dfrac{1}{|B_{\gamma} (x)|_{h}^{1-\alpha/n}}  \dint_{B_{\gamma} (x)} |b(x)-b(y)| |f(y)| \mathd y,
\end{align*}
where the supremum is taken over all $p$-adic balls $B_{\gamma} (x)\subset \mathbb{Q}_{p}^{n}$.
 \item  The nonlinear commutators generated by   $\mathcal{M}_{\alpha}^{p}$ and $b$  is defined by
\begin{align*}
[b,\mathcal{M}_{\alpha}^{p}] (f)(x) &= b(x) \mathcal{M}_{\alpha}^{p} (f)(x) -\mathcal{M}_{\alpha}^{p}(bf)(x).
\end{align*}
\end{enumerate}
\end{definition}

When $\alpha=0$, we simply denote by $[b,\mathcal{M}^{p}]=[b,\mathcal{M}_{0}^{p}]$ and $\mathcal{M}_{b}^{p}=\mathcal{M}_{0,b}^{p}$.

We call $[b,\mathcal{M}_{\alpha}^{p}] $ the nonlinear commutator because it is not even a sublinear operator,
although the commutator $[b,T]$ is a linear one. It is worth noting that the nonlinear commutator $[b,\mathcal{M}_{\alpha}^{p}] $ and the maximal commutator $\mathcal{M}_{\alpha,b}^{p}$ essentially differ from each other.
For example, $\mathcal{M}_{\alpha,b}^{p}$ is positive and sublinear, but $[b,\mathcal{M}_{\alpha}^{p}] $ is neither positive nor sublinear

Denote by $M$  and $M_{\alpha}$ the classical Hardy–Littlewood maximal function and the   fractional maximal function in $\mathbb{R}^{n}$ respectively.
In fact,  the nonlinear commutator $[b, M]$ and  $[b, M_{\alpha}]$ have been studied by many authors in the Euclidean spaces, for instance, \cite{milman1990second,bastero2000commutators,zhang2009commutators,zhang2014commutatorsfor,zhang2014commutators,agcayazi2015note,guliyev2017fractional,zhang2017characterization,zhang2018commutators,zhang2019characterization}
  etc.
When the symbol $b$ belongs to $\bmo(\mathbb{R}^{n})$, \citet{bastero2000commutators} studied the necessary and sufficient conditions for   the boundedness of $[b,M]$ in $L^{q}(\mathbb{R}^{n})$ for $1<q<\infty$. Zhang and Wu obtained similar results for the fractional maximal function in \cite{zhang2009commutators} and extended the mentioned results to variable exponent Lebesgue spaces in \cite{zhang2014commutatorsfor,zhang2014commutators}.
When the symbol $b$ belongs to Lipschitz spaces,   \citet{zhang2018commutators,zhang2019some} gave the necessary and sufficient conditions for the boundedness of  $[b,M_{\alpha}]$ on Orlicz spaces and variable Lebesgue spaces respectively.
And recently,  \citet{yang2023characterization} considered some new characterizations of a variable version of Lipschitz spaces in terms of the boundedness of commutators of fractional maximal functions or fractional maximal commutators  in the context of the variable Lebesgue spaces.

On the other hand, \citet{he2022characterization} gave the characterization of  $p$-adic Lipschitz spaces in terms of the boundedness of commutators of maximal function $\mathcal{M}^{p}$ in the context of the $p$-adic Lebesgue spaces and Morrey spaces  when the symbols $b$ belong to  $p$-adic Lipschitz spaces $\Lambda_{\beta}(\mathbb{Q}_{p}^{n})$. And \citet{chacon2021fractional} prove the boundedness of the fractional maximal and the fractional integral operator in the $p$-adic variable exponent Lebesgue spaces.

Inspired by the above literature, we focus on the case of $p$-adic fields $\mathbb{Q}_{p}$, in some sense it can also be pointed out  that our work was motivated by the standard harmonic analysis on the Euclidean space, the purpose of this paper is to   study the boundedness of the $p$-adic fractional maximal commutator $ \mathcal{M}_{\alpha,b}^{p}$ or the nonlinear commutator $[b,\mathcal{M}_{\alpha}^{p}] $ generated by $p$-adic fractional maximal function $ \mathcal{M}_{\alpha}^{p} $  over  $p$-adic  variable exponent  Lebesgue spaces, where the symbols $b$ belong to the $p$-adic Lipschitz spaces, by which some new characterizations of the $p$-adic version of Lipschitz spaces are given.

Let $\alpha\ge 0$, for a fixed $p$-adic ball $B_{*}$, the fractional maximal function with respect to $B_{*}$ of a locally integrable function $f$ is given by
\begin{align*}
  \mathcal{M}_{\alpha,B_{*}}^{p}(f)(x) = \sup_{\gamma\in \mathbb{Z}  \atop B_{\gamma} (x)\subset B_{*}} \dfrac{1}{|B_{\gamma} (x)|_{h}^{1-\alpha/n}} \dint_{B_{\gamma} (x)} |f(y)| \mathd y,
\end{align*}
where the supremum is taken over all the $p$-adic  ball $B_{\gamma} (x)$ with $B_{\gamma} (x)\subset B_{*}$ for a fixed $p$-adic  ball $B_{*}$. When $\alpha= 0$,  we simply write $\mathcal{M}_{B_{*}}^{p}$ instead of $\mathcal{M}_{0,B_{*}}^{p}$.

Our main results can be stated as follows,
 which are to study the boundedness of $\mathcal{M}_{\alpha,b}^{p}$ and $[b,\mathcal{M}_{\alpha}^{p}] $ in the context of $p$-adic  variable exponent Lebesgue spaces when the symbol belongs to a  $p$-adic version of Lipschitz spaces $ \Lambda_{\beta}(\mathbb{Q}_{p}^{n})$ (see   \cref{sec:preliminary}  below). And  some new characterizations of the Lipschitz spaces  via such commutators are given.

\begin{theorem} \label{thm:nonlinear-frac-max-lip}  
 Let  $0 <\beta <1$, $0< \alpha<\alpha+\beta<n$  and $b$ be a locally integrable function on $\mathbb{Q}_{p}^{n}$.
 Then the following assertions are equivalent:
\begin{enumerate}[label=(A.\arabic*),align=left,itemindent=1em]   
\item   $b\in  \Lambda_{\beta}(\mathbb{Q}_{p}^{n})$ and $b\ge 0$.
    \label{enumerate:thm-nonlinear-frac-max-lip-1}
   \item The commutator $ [b,\mathcal{M}_{\alpha}^{p}] $ is bounded from $L^{r(\cdot)}(\mathbb{Q}_{p}^{n})$ to $L^{q(\cdot)}(\mathbb{Q}_{p}^{n})$ for all $r(\cdot), q(\cdot)\in   \mathscr{C}^{\log}(\mathbb{Q}_{p}^{n}) $ with $r(\cdot)\in \mathscr{P}(\mathbb{Q}_{p}^{n})$, $ r_{+}<\frac{n}{\alpha+\beta}$ and $1/q(\cdot) = 1/r(\cdot) -(\alpha+\beta)/n$.
    \label{enumerate:thm-nonlinear-frac-max-lip-2}
\item  The commutator $ [b,\mathcal{M}_{\alpha}^{p}] $ is bounded from $L^{r(\cdot)}(\mathbb{Q}_{p}^{n})$ to $L^{q(\cdot)}(\mathbb{Q}_{p}^{n})$  for some  $r(\cdot), q(\cdot)\in   \mathscr{C}^{\log}(\mathbb{Q}_{p}^{n}) $ with $r(\cdot)\in \mathscr{P}(\mathbb{Q}_{p}^{n})$, $r_{+}<\frac{n}{\alpha+\beta}$ and $1/q(\cdot) = 1/r(\cdot) -(\alpha+\beta)/n$.
    \label{enumerate:thm-nonlinear-frac-max-lip-3}
   \item There exists some  $r(\cdot), q(\cdot)\in   \mathscr{C}^{\log}(\mathbb{Q}_{p}^{n}) $ with $r(\cdot)\in \mathscr{P}(\mathbb{Q}_{p}^{n})$,  $r_{+}<\frac{n}{\alpha+\beta}$ and $1/q(\cdot) = 1/r(\cdot) -(\alpha+\beta)/n$, such that
\begin{align} \label{inequ:thm-nonlinear-frac-max-lip-4}
\sup_{\gamma\in \mathbb{Z} \atop x\in \mathbb{Q}_{p}^{n}} \dfrac{1}{|B_{\gamma} (x)|_{h}^{\beta/n}} \dfrac{\Big\| \big(b -|B_{\gamma} (x)|_{h}^{-\alpha/n}\mathcal{M}_{\alpha,B_{\gamma} (x)}^{p} (b) \big) \dchi_{B_{\gamma} (x)} \Big\|_{L^{q(\cdot)}(\mathbb{Q}_{p}^{n}) }}{\|\dchi_{B_{\gamma} (x)}\|_{L^{q(\cdot)}(\mathbb{Q}_{p}^{n}) }}  < \infty.
\end{align}
    \label{enumerate:thm-nonlinear-frac-max-lip-4}
   \item  For all   $r(\cdot), q(\cdot)\in   \mathscr{C}^{\log}(\mathbb{Q}_{p}^{n}) $ with $r(\cdot)\in \mathscr{P}(\mathbb{Q}_{p}^{n})$,  $r_{+}<\frac{n}{\alpha+\beta}$ and $1/q(\cdot) = 1/r(\cdot) -(\alpha+\beta)/n$, such that \labelcref{inequ:thm-nonlinear-frac-max-lip-4} holds.
\label{enumerate:thm-nonlinear-frac-max-lip-5}
\end{enumerate}
\end{theorem}

For the case of $r(\cdot)$ and $q(\cdot)$ being constants, we have the following result from \cref{thm:nonlinear-frac-max-lip}, which is new even for this case.

\begin{corollary}  \label{cor:nonlinear-frac-max-lip}
 Let $0 <\beta <1$, $0< \alpha<\alpha+\beta<n$ and $b$ be a locally integrable function  on $\mathbb{Q}_{p}^{n}$. Then the following statements are equivalent:
\begin{enumerate}[label=(C.\arabic*),itemindent=2em]
\item   $b\in  \Lambda_{\beta}(\mathbb{Q}_{p}^{n})$  and $b\ge 0$.
    \label{enumerate:cor-nonlinear-frac-max-lip-1}
   \item The commutator $[b,\mathcal{M}_{\alpha}^{p}]$ is bounded from $L^{r}(\mathbb{Q}_{p}^{n})$ to $L^{q}(\mathbb{Q}_{p}^{n})$ for all $r, q$ with $1<r<\frac{n}{\alpha+\beta}$ and $1/q = 1/r -(\alpha+\beta)/n$.
    \label{enumerate:cor-nonlinear-frac-max-lip-2}
   \item  The commutator $[b,\mathcal{M}_{\alpha}^{p}]$ is bounded from $L^{r}(\mathbb{Q}_{p}^{n})$ to $L^{q}(\mathbb{Q}_{p}^{n})$ for some $r, q$ with $1<r<\frac{n}{\alpha+\beta}$ and $1/q = 1/r -(\alpha+\beta)/n$.
    \label{enumerate:cor-nonlinear-frac-max-lip-3}
   \item  There exists some  $r, q$ with $1<r<\frac{n}{\alpha+\beta}$ and $1/q = 1/r -(\alpha+\beta)/n$, such that
\begin{align} \label{inequ:cor-nonlinear-frac-max-lip-4}
 \sup_{\gamma\in \mathbb{Z} \atop x\in \mathbb{Q}_{p}^{n}} \dfrac{1}{|B_{\gamma} (x)|_{h}^{\beta/n}} \Big(\dfrac{1}{|B_{\gamma} (x)|_{h}} \dint_{B_{\gamma} (x)} \Big|b(y)-|B_{\gamma} (x)|_{h}^{-\alpha/n} \mathcal{M}_{\alpha,B_{\gamma} (x)}^{p} (b)(y) \Big|^{q} \mathd y \Big)^{1/q} <\infty.
\end{align}
\label{enumerate:cor-nonlinear-frac-max-lip-4}
   \item  For all   $r, q$ with $1<r<\frac{n}{\alpha+\beta}$ and $1/q = 1/r -(\alpha+\beta)/n$, such that \labelcref{inequ:cor-nonlinear-frac-max-lip-4} holds.
\label{enumerate:cor-nonlinear-frac-max-lip-5}
\end{enumerate}
\end{corollary}

\begin{remark}   \label{rem.cor-nonlinear-frac-max-lip}
\begin{enumerate}[ label=(\roman*)]  
\item For the case $\alpha=0$, the partial results of  
    \cref{cor:nonlinear-frac-max-lip} were given in \cite{he2022characterization} (see Theorem 4).
\item  Moreover, it was proved in Theorem 4 of  \cite{he2022characterization}, see also \cref{lem:non-negative-max-lip} below, that  $b\in  \Lambda_{\beta}(\mathbb{Q}_{p}^{n})$  and $b\ge 0$ if and only if
\begin{align} \label{inequ:12-he2022characterization}
 \sup_{\gamma\in \mathbb{Z} \atop x\in \mathbb{Q}_{p}^{n}} \dfrac{1}{|B_{\gamma} (x)|_{h}^{\beta/n}} \Big(\dfrac{1}{|B_{\gamma} (x)|_{h}} \dint_{B_{\gamma} (x)} \Big|b(y)-  \mathcal{M}_{B_{\gamma} (x)}^{p} (b)(y) \Big|^{q} \mathd y \Big)^{1/q} <\infty.
\end{align}
Compared with \labelcref{inequ:12-he2022characterization},  \labelcref{inequ:cor-nonlinear-frac-max-lip-4} gives a new characterization for nonnegative Lipschitz functions.
\end{enumerate}
\end{remark}

In particular, when $\alpha=0$, the results are also true  come from \cref{thm:nonlinear-frac-max-lip} and \cref{cor:nonlinear-frac-max-lip}.
Now we only give the case  in   the context of the $p$-adic version of  variable exponent  Lebesgue spaces, and it is new.

\begin{corollary}  \label{cor:nonlinear-max-lip-variable}
 Let $0 <\beta <1$ and $b$ be a locally integrable function  on $\mathbb{Q}_{p}^{n}$. Then the following statements are equivalent:
\begin{enumerate}[label=(C.\arabic*),itemindent=2em]
\item   $b\in  \Lambda_{\beta}(\mathbb{Q}_{p}^{n})$  and $b\ge 0$.
    \label{enumerate:cor-nonlinear-max-lip-variable-1}
   \item The commutator $ [b,\mathcal{M}^{p}] $ is bounded from $L^{r(\cdot)}(\mathbb{Q}_{p}^{n})$ to $L^{q(\cdot)}(\mathbb{Q}_{p}^{n})$ for all $r(\cdot), q(\cdot)\in   \mathscr{C}^{\log}(\mathbb{Q}_{p}^{n}) $ with $r(\cdot)\in \mathscr{P}(\mathbb{Q}_{p}^{n})$, $ r_{+}<\frac{n}{\beta}$ and $1/q(\cdot) = 1/r(\cdot) - \beta/n$.
    \label{enumerate:cor-nonlinear-max-lip-variable-2}
\item  The commutator $ [b,\mathcal{M}^{p}] $ is bounded from $L^{r(\cdot)}(\mathbb{Q}_{p}^{n})$ to $L^{q(\cdot)}(\mathbb{Q}_{p}^{n})$  for some  $r(\cdot), q(\cdot)\in   \mathscr{C}^{\log}(\mathbb{Q}_{p}^{n}) $ with $r(\cdot)\in \mathscr{P}(\mathbb{Q}_{p}^{n})$, $r_{+}<\frac{n}{\beta}$ and $1/q(\cdot) = 1/r(\cdot) -\beta/n$.
    \label{enumerate:cor-nonlinear-max-lip-variable-3}
   \item There exists some  $r(\cdot), q(\cdot)\in   \mathscr{C}^{\log}(\mathbb{Q}_{p}^{n}) $ with $r(\cdot)\in \mathscr{P}(\mathbb{Q}_{p}^{n})$,  $r_{+}<\frac{n}{\beta}$ and $1/q(\cdot) = 1/r(\cdot) -\beta/n$, such that
\begin{align} \label{inequ:cor-nonlinear-max-lip-variable-4}
\sup_{\gamma\in \mathbb{Z} \atop x\in \mathbb{Q}_{p}^{n}} \dfrac{1}{|B_{\gamma} (x)|_{h}^{\beta/n}} \dfrac{\Big\| \big(b - \mathcal{M}_{B_{\gamma} (x)}^{p} (b) \big) \dchi_{B_{\gamma} (x)} \Big\|_{L^{q(\cdot)}(\mathbb{Q}_{p}^{n}) }}{\|\dchi_{B_{\gamma} (x)}\|_{L^{q(\cdot)}(\mathbb{Q}_{p}^{n}) }}  < \infty.
\end{align}
    \label{enumerate:cor-nonlinear-max-lip-variable-4}
   \item  For all   $r(\cdot), q(\cdot)\in   \mathscr{C}^{\log}(\mathbb{Q}_{p}^{n}) $ with $r(\cdot)\in \mathscr{P}(\mathbb{Q}_{p}^{n})$,  $r_{+}<\frac{n}{\beta}$ and $1/q(\cdot) = 1/r(\cdot) - \beta/n$, such that \labelcref{inequ:cor-nonlinear-max-lip-variable-4} holds.
\label{enumerate:cor-nonlinear-max-lip-variable-5}
\end{enumerate}
\end{corollary}

\begin{theorem} \label{thm:frac-max-lip}
 Let $0 <\beta <1$, $0< \alpha<\alpha+\beta<n$ and $b$ be a locally integrable function on $\mathbb{Q}_{p}^{n}$.
 Then the following assertions are equivalent:
\begin{enumerate}[label=(B.\arabic*),align=left,itemindent=1em]   
\item   $b\in  \Lambda_{\beta}(\mathbb{Q}_{p}^{n})$.
    \label{enumerate:thm-frac-max-lip-1}
   \item The commutator $ \mathcal{M}_{\alpha,b}^{p}$ is bounded from $L^{r(\cdot)}(\mathbb{Q}_{p}^{n})$ to $L^{q(\cdot)}(\mathbb{Q}_{p}^{n})$ for all $r(\cdot), q(\cdot)\in   \mathscr{C}^{\log}(\mathbb{Q}_{p}^{n}) $ with $r(\cdot)\in \mathscr{P}(\mathbb{Q}_{p}^{n})$, $ r_{+}<\frac{n}{\alpha+\beta}$ and $1/q(\cdot) = 1/r(\cdot) -(\alpha+\beta)/n$.
    \label{enumerate:thm-frac-max-lip-2}
\item  The commutator $ \mathcal{M}_{\alpha,b}^{p}$  is bounded from $L^{r(\cdot)}(\mathbb{Q}_{p}^{n})$ to $L^{q(\cdot)}(\mathbb{Q}_{p}^{n})$  for some  $r(\cdot), q(\cdot)\in   \mathscr{C}^{\log}(\mathbb{Q}_{p}^{n}) $ with $r(\cdot)\in \mathscr{P}(\mathbb{Q}_{p}^{n})$,  $ r_{+}<\frac{n}{\alpha+\beta}$ and $1/q(\cdot) = 1/r(\cdot) -(\alpha+\beta)/n$.
    \label{enumerate:thm-frac-max-lip-3}
   \item There exists some  $r(\cdot), q(\cdot)\in   \mathscr{C}^{\log}(\mathbb{Q}_{p}^{n}) $ with $r(\cdot)\in \mathscr{P}(\mathbb{Q}_{p}^{n})$,   $ r_{+}<\frac{n}{\alpha+\beta}$ and $1/q(\cdot) = 1/r(\cdot) -(\alpha+\beta)/n$, such that
\begin{align} \label{inequ:thm-frac-max-lip-4}
 \sup_{\gamma\in \mathbb{Z} \atop x\in \mathbb{Q}_{p}^{n}} \dfrac{1}{|B_{\gamma} (x)|_{h}^{\beta/n}} \dfrac{\Big\| \big(b -b_{B_{\gamma}(x)} \big) \dchi_{B_{\gamma} (x)} \Big\|_{L^{q(\cdot)}(\mathbb{Q}_{p}^{n}) }}{\|\dchi_{B_{\gamma} (x)}\|_{L^{q(\cdot)}(\mathbb{Q}_{p}^{n}) }}  < \infty.
\end{align}
    \label{enumerate:thm-frac-max-lip-4}
   \item  For all   $r(\cdot), q(\cdot)\in   \mathscr{C}^{\log}(\mathbb{Q}_{p}^{n}) $ with $r(\cdot)\in \mathscr{P}(\mathbb{Q}_{p}^{n})$,   $ r_{+}<\frac{n}{\alpha+\beta}$ and $1/q(\cdot) = 1/r(\cdot) -(\alpha+\beta)/n$, such that \labelcref{inequ:thm-frac-max-lip-4} holds.
\label{enumerate:thm-frac-max-lip-5}
\end{enumerate}
\end{theorem}

When  $r(\cdot)$ and $q(\cdot)$ are constants, we get the following result from \cref{thm:frac-max-lip}.

\begin{corollary}  \label{cor:frac-max-lip}
 Let $0 <\beta <1$, $0< \alpha<\alpha+\beta<n$ and $b$ be a locally integrable function on $\mathbb{Q}_{p}^{n}$. Then the following statements are equivalent:
\begin{enumerate}[label=(C.\arabic*),itemindent=2em]
\item   $b\in  \Lambda_{\beta}(\mathbb{Q}_{p}^{n})$.
    \label{enumerate:cor-frac-max-lip-1}
   \item The commutator $\mathcal{M}_{\alpha,b}^{p}$ is bounded from $L^{r}(\mathbb{Q}_{p}^{n})$ to $L^{q}(\mathbb{Q}_{p}^{n})$ for all $r, q$ with $1<r<\frac{n}{\alpha+\beta}$ and $1/q = 1/r -(\alpha+\beta)/n$.
    \label{enumerate:cor-frac-max-lip-2}
\item  The commutator $\mathcal{M}_{\alpha,b}^{p}$  is bounded from $L^{r}(\mathbb{Q}_{p}^{n})$ to $L^{q}(\mathbb{Q}_{p}^{n})$ for some $r, q$ with $1<r<\frac{n}{\alpha+\beta}$ and $1/q = 1/r -(\alpha+\beta)/n$.
    \label{enumerate:cor-frac-max-lip-3}
   \item There exists some  $r, q$ with $1<r<\frac{n}{\alpha+\beta}$ and $1/q = 1/r -(\alpha+\beta)/n$, such that
\begin{align} \label{inequ:cor-frac-max-lip-4}
  \sup_{\gamma\in \mathbb{Z} \atop x\in \mathbb{Q}_{p}^{n}} \dfrac{1}{|B_{\gamma} (x)|_{h}^{\beta/n}} \Big(\dfrac{1}{|B_{\gamma} (x)|_{h}} \dint_{B_{\gamma} (x)} \Big|b(y)-b_{B_{\gamma}(x)} \Big|^{q} \mathd y \Big)^{1/q} <\infty.
\end{align}
    \label{enumerate:cor-frac-max-lip-4}
   \item  For all   ${r, q}$ with  $1<r<\frac{n}{\alpha+\beta}$ and $1/q = 1/r -(\alpha+\beta)/n$, such that \labelcref{inequ:cor-frac-max-lip-4} holds.
\label{enumerate:cor-frac-max-lip-5}
\end{enumerate}
\end{corollary}

\begin{remark}   \label{rem.cor-frac-max-lip}
\begin{enumerate}[ label=(\roman*)]
\item   For the case $\alpha=0$,  \cref{cor:frac-max-lip} is also holds, and the equivalence of \labelcref{enumerate:cor-frac-max-lip-1}, \labelcref{enumerate:cor-frac-max-lip-2} and \labelcref{enumerate:cor-frac-max-lip-3}  was proved in \cite{he2022characterization} (see Theorem 1).
\item  Moreover,  the equivalence of \labelcref{enumerate:cor-frac-max-lip-1}, \labelcref{enumerate:cor-frac-max-lip-4} and \labelcref{enumerate:cor-frac-max-lip-5}  is contained in \cref{lem:6-he2022characterization} below.
\end{enumerate}
\end{remark}

Finally, we give the follows result, which is valid and new,   from  \cref{thm:frac-max-lip} with   $\alpha=0$.
\begin{corollary}  \label{cor:frac-max-lip-variable}
 Let $0 <\beta <1$ and $b$ be a locally integrable function  on $\mathbb{Q}_{p}^{n}$. Then the following statements are equivalent:
\begin{enumerate}[label=(C.\arabic*),itemindent=2em]
\item   $b\in  \Lambda_{\beta}(\mathbb{Q}_{p}^{n})$.
    \label{enumerate:cor-frac-max-lip-variable-1}
   \item The commutator $ \mathcal{M}_{b}^{p}$ is bounded from $L^{r(\cdot)}(\mathbb{Q}_{p}^{n})$ to $L^{q(\cdot)}(\mathbb{Q}_{p}^{n})$ for all $r(\cdot), q(\cdot)\in   \mathscr{C}^{\log}(\mathbb{Q}_{p}^{n}) $ with $r(\cdot)\in \mathscr{P}(\mathbb{Q}_{p}^{n})$, $ r_{+}<\frac{n}{ \beta}$ and $1/q(\cdot) = 1/r(\cdot) - \beta/n$.
    \label{enumerate:cor-frac-max-lip-variable-2}
\item  The commutator $ \mathcal{M}_{b}^{p}$  is bounded from $L^{r(\cdot)}(\mathbb{Q}_{p}^{n})$ to $L^{q(\cdot)}(\mathbb{Q}_{p}^{n})$  for some  $r(\cdot), q(\cdot)\in   \mathscr{C}^{\log}(\mathbb{Q}_{p}^{n}) $ with $r(\cdot)\in \mathscr{P}(\mathbb{Q}_{p}^{n})$,  $ r_{+}<\frac{n}{\beta}$ and $1/q(\cdot) = 1/r(\cdot) -\beta/n$.
    \label{enumerate:cor-frac-max-lip-variable-3}
   \item There exists some  $r(\cdot), q(\cdot)\in   \mathscr{C}^{\log}(\mathbb{Q}_{p}^{n}) $ with $r(\cdot)\in \mathscr{P}(\mathbb{Q}_{p}^{n})$,   $ r_{+}<\frac{n}{\beta}$ and $1/q(\cdot) = 1/r(\cdot) -\beta/n$, such that
\begin{align} \label{inequ:cor-frac-max-lip-variable-4}
 \sup_{\gamma\in \mathbb{Z} \atop x\in \mathbb{Q}_{p}^{n}} \dfrac{1}{|B_{\gamma} (x)|_{h}^{\beta/n}} \dfrac{\Big\| \big(b -b_{B_{\gamma}(x)} \big) \dchi_{B_{\gamma} (x)} \Big\|_{L^{q(\cdot)}(\mathbb{Q}_{p}^{n}) }}{\|\dchi_{B_{\gamma} (x)}\|_{L^{q(\cdot)}(\mathbb{Q}_{p}^{n}) }}  < \infty.
\end{align}
    \label{enumerate:cor-frac-max-lip-variable-4}
   \item  For all   $r(\cdot), q(\cdot)\in   \mathscr{C}^{\log}(\mathbb{Q}_{p}^{n}) $ with $r(\cdot)\in \mathscr{P}(\mathbb{Q}_{p}^{n})$,   $ r_{+}<\frac{n}{\beta}$ and $1/q(\cdot) = 1/r(\cdot) -\beta/n$, such that \labelcref{inequ:cor-frac-max-lip-variable-4} holds.
\label{enumerate:cor-frac-max-lip-variable-5}
\end{enumerate}
\end{corollary}

 Throughout this paper, the letter $C$  always stands for a constant  independent of the main parameters involved and whose value may differ from line to line.
In addition, we  give some notations. Here and hereafter $|E|_{h}$  will always denote the Haar measure of a measurable set $E$ on $\mathbb{Q}_{p}^{n}$ and by  \raisebox{2pt}{$\dchi_{E}$} denotes the  characteristic function of a measurable set $E\subset\mathbb{Q}_{p}^{n}$.


\section{Preliminaries and lemmas}
\label{sec:preliminary}

To prove the main results of this paper, we first recall some necessary notions and remarks.

\subsection{$p$-adic field $\mathbb{Q}_{p}$}

Firstly, we introduce some basic and necessary notations for the $p$-adic field.

Let $p \ge 2$ be a fixed prime number in $\mathbb{Z}$ and $G_{p}=\{0,1,\ldots,p-1\}$.
For every non-zero rational number $x$, by the unique factorization theorem, there is a unique $\gamma=\gamma(x)\in \mathbb{Z}$, such that $x=p^{\gamma} \frac{m}{n}$, where $m,n\in \mathbb{Z}$  are not divisible by $p$ (i.e. $p$ is coprime to $m$, $n$).
Define the mapping $|\cdot|_{p}: \mathbb{Q} \to \mathbb{R}^{+}$  as follows:
\begin{align*} 
|x|_{p}=
\begin{cases}
p^{-\gamma} & \text{if} \ x \neq 0,
\\
0 & \text{if} \ x = 0.
\end{cases}
\end{align*}
 The $p$-adic absolute value $|\cdot|_{p}$ is endowed with many properties of the usual real norm $|\cdot|$ with an
additional non-Archimedean property (i.e., $\{|m|_{p}, m\in \mathbb{Z}\}$ is bounded)
$$|x+y|_{p} \le \max \{ |x|_{p},|y|_{p}\}.$$
In addition,  $|\cdot|_{p}$  also satisfies the following properties:
\begin{enumerate}[label=(\arabic*),itemindent=2em]
\item  (positive definiteness) $|x|_{p} \ge 0$. Specially,  $|x|_{p}=0 \Leftrightarrow  ~x = 0$
\item (multiplicativity) $|xy|_{p} = |x|_{p} |y|_{p} $.
\item (non-Archimedean triangle inequality)  $|x+y|_{p} \le \max\{ |x|_{p}, |y|_{p}\} $. The equality holds if and only if $|x|_{p} \neq |y|_{p}$.
\end{enumerate}
 Denote by $\mathbb{Q}_{p}$ the $p$-adic field which is defined as the completion of the field of rational numbers $\mathbb{Q}$ with respect to the   $p$-adic absolute value $|\cdot|_{p}$.

From the standard $p$-adic analysis, any non-zero element $x\in\mathbb{Q}_{p}$ can be  uniquely represented as a canonical series form
\begin{align*}
 x=p^{\gamma}(a_{0}+a_{1}p+a_{2}p^{2}+\cdots)  =p^{\gamma}   \sum_{j=0}^{\infty} a_{j}p^{j},
\end{align*}
where $a_{j}\in G_{p}$ and $a_{0}\neq 0$, and  $\gamma =\gamma(x)\in \mathbb{Z}$ is called as the $p$-adic valuation of $x$. The series converges in the $p$-adic  absolute value since the inequality $| a_{j}p^{j}|_{p} \le  p^{-j}$  holds for all   $j\in \mathbb{N}$.

Moreover, the $n$-dimensional  $p$-adic  vector space $\mathbb{Q}_{p}^{n} =\mathbb{Q}_{p}\times\cdots\times \mathbb{Q}_{p}~(n\ge 1)$, consists of all points $x= (x_{1},\ldots,x_{n})$, where $x_{i} \in \mathbb{Q}_{p} ~(i=1,\ldots,n)$,  equipped with the following absolute value
\begin{align*}
 |x|_{p}= \max_{1\le j\le n} |x_{j}|_{p}.
\end{align*}
For $\gamma\in \mathbb{Z}$ and $a= (a_{1},a_{2},\dots,a_{n})\in \mathbb{Q}_{p}^{n}$, we denote by
\begin{align*}
  B_{\gamma} (a)&=    \{ x\in \mathbb{Q}_{p}^{n}: |x-a|_{p}\le p^{\gamma}\}
\end{align*}
the closed ball with the center at $a$ and radius $p^{\gamma}$ and by
\begin{align*}
  S_{\gamma} (a)&=    \{ x\in  \mathbb{Q}_{p}^{n}: |x-a|_{p}= p^{\gamma}\}  =B_{\gamma}(a)\setminus B_{\gamma-1}(a)
\end{align*}
the corresponding sphere.
For $a=0$, we write $B_{\gamma} (0) = B_{\gamma} $, and $S_{\gamma} (0) = S_{\gamma} $.
Note that $B_{\gamma} (a) =\bigcup\limits_{k\le \gamma} S_{k} (a)$ 
and $\mathbb{Q}_{p}^{n}\setminus \{0\}= \bigcup\limits_{\gamma\in \mathbb{Z}} S_{\gamma} $.
It is easy to see that the equalities
\begin{align*}
 a_{0} +B_{\gamma} = B_{\gamma} (a_{0})
 \ \text{and} \
 a_{0} +S_{\gamma} = S_{\gamma} (a_{0}) = B_{\gamma} (a_{0}) \setminus B_{\gamma-1} (a_{0})
\end{align*}
hold for all $a_{0} \in \mathbb{Q}_{p}^{n}$ and $\gamma\in \mathbb{Z}$.

It follows from non-Archimedean triangle inequality that two balls $B_{\gamma} (x)$ and $B_{\gamma'} (y)$ either do not intersect or one of
them is contained in the other, which differ from those of the Euclidean case. And note that in the second case under conditions $\gamma=\gamma'$ these balls are equal.
 The above properties can also be found in \cite{kim2009q} (see Lemma 3.1).
\begin{lemma}\label{lem:lem-3.1-kim2009q}
 Let $\gamma, \gamma'\in \mathbb{Z}$,  $x,y\in  \mathbb{Q}_{p}^{n}$. The $p$-adic balls have the following properties:
\begin{enumerate}[label=(\arabic*),itemindent=1em] 
\item If   $\gamma\le\gamma'$, then  either $B_{\gamma} (x) \cap B_{\gamma'} (y)=\emptyset$ or $B_{\gamma} (x) \subset B_{\gamma'} (y)$
\item $B_{\gamma} (x)= B_{\gamma} (y)$   if and only if  $y\in B_{\gamma} (x)$.
\end{enumerate}
\end{lemma}

Since $\mathbb{Q}_{p}^{n}$ is a locally compact commutative group with respect to addition, there exists a unique
Haar measure $\mathd x$  on $\mathbb{Q}_{p}^{n}$ (up to positive constant multiple) which is translation invariant (i.e., $\mathd (x+a )= \mathd x$), such that
\begin{align*}
  \dint_{B_{0}} \mathd x &=  |B_{0}|_{h} =1,
\end{align*}
where $|E|_{h}$ denotes the Haar measure of measurable subset $E $ of $\mathbb{Q}_{p}^{n}$. Furthermore, from this integral theory, it is
easy to obtain that 
\begin{align}  \label{equ:p-adic-integral-theory}  
  \dint_{B_{\gamma}(a)} \mathd x &=|B_{\gamma}(a)|_{h}      = p^{n\gamma}
\\ \intertext{and}
 \dint_{S_{\gamma}(a)} \mathd x &= |S_{\gamma}(a)|_{h}     =  p^{n\gamma}(1-p^{-n})= |B_{\gamma}(a)|_{h} -|B_{\gamma-1}(a)|_{h}      \notag
\end{align}
hold for all $a \in \mathbb{Q}_{p}^{n}$ and $\gamma\in \mathbb{Z}$.

For more information about the p-adic field, we refer readers to \cite{taibleson1975fourier,vladimirov1994padic}.

\subsection{$p$-adic function spaces}

In what follows, we say that a  real-valued  measurable
function $f$ defined on $\mathbb{Q}_{p}^{n}$ is in $L^{q}(\mathbb{Q}_{p}^{n})$, $1\le q\le \infty$, if it satisfies
\begin{align} \label{equ:4-he2022characterization}
\|f\|_{L^{q}(\mathbb{Q}_{p}^{n})} &=  \Big(\dint_{\mathbb{Q}_{p}^{n}}  |f(x)|^{q} \mathd x \Big)^{1/q} <\infty,\ \ 1\le q<\infty
\end{align}
and denote by $L^{\infty}(\mathbb{Q}_{p}^{n})$   the set of all measurable real-valued functions  $f$ on $\mathbb{Q}_{p}^{n}$ satisfying
\begin{align*}
  \|f\|_{L^{\infty}(\mathbb{Q}_{p}^{n})}  &=\esssup_{x\in \mathbb{Q}_{p}^{n}} |f(x)|
<\infty.
\end{align*}
Here, 
the integral in equation \labelcref{equ:4-he2022characterization} is defined as follows:
\begin{align*} 
\begin{aligned}
   \dint_{\mathbb{Q}_{p}^{n}}  |f(x)|^{q} \mathd x  &=\lim_{\gamma\to\infty} \dint_{B_{\gamma}(0)}  |f(x)|^{q} \mathd x
   =\lim_{\gamma\to\infty} \sum_{-\infty<k\le \gamma} \dint_{S_{k}(0)}  |f(x)|^{q} \mathd x,
\end{aligned}
\end{align*}
if the limit exists.

 Some often used computational principles are worth noting. In particular,  if $f\in  L^{1}(\mathbb{Q}_{p}^{n}) $, then
\begin{align*}
   \dint_{\mathbb{Q}_{p}^{n}}  f(x)  \mathd x  &= \sum_{\gamma=-\infty}^{+\infty} \dint_{S_{\gamma}}  f(x)  \mathd x
\\ \intertext{and}
 \dint_{\mathbb{Q}_{p}^{n}} f(tx)  \mathd x  &= \frac{1}{|t|_{p}^{n}} \dint_{\mathbb{Q}_{p}^{n}} f(x) \mathd x,
\end{align*}
where $t\in \mathbb{Q}_{p}\setminus\{0\}$, $tx= (tx_{1},\ldots,tx_{n})$ and $\mathd (tx)=|t|_{p}^{n} \mathd x$.

We now introduce the notion of $p$-adic variable exponent Lebesgue spaces and give some properties needed in
the sequel (see \cite{chacon2020variable} for the respective proofs).

We say that a measurable function $q(\cdot)$ is a variable exponent if $q(\cdot): \mathbb{Q}_{p}^{n}\to (0,\infty)$.

\begin{definition} \label{def.variable-exponent}
  Given a measurable function $q(\cdot)$ defined on $\mathbb{Q}_{p}^{n}$,   we denote by
$$
q_{-} :=\essinf_{x\in \mathbb{Q}_{p}^{n}} q(x),\ \
q_{+}:= \esssup_{x\in \mathbb{Q}_{p}^{n}} q(x).$$
\begin{enumerate}[label=(\arabic*),itemindent=1em] 
\item $q'_{-}=\essinf\limits_{x\in \mathbb{Q}_{p}^{n}} q'(x)=\frac{q_{+}}{q_{+}-1},\ \ q'_{+}= \esssup\limits_{x\in \mathbb{Q}_{p}^{n}} q'(x)=\frac{q_{-}}{q_{-}-1}.$
 \item  Denote by $\mathscr{P}_{0}(\mathbb{Q}_{p}^{n})$ the set of all measurable functions $ q(\cdot): \mathbb{Q}_{p}^{n}\to(0,\infty)$ such that
$$0< q_{-}\le q(x) \le q_{+}<\infty,\ \ x\in \mathbb{Q}_{p}^{n}.$$
\item  Denote by $\mathscr{P}_{1}(\mathbb{Q}_{p}^{n})$ the set of all measurable functions $ q(\cdotp): \mathbb{Q}_{p}^{n}\to[1,\infty)$ such that
$$1\le q_{-}\le q(x) \le q_{+}<\infty,\ \ x\in \mathbb{Q}_{p}^{n}.$$
  \item Denote by $\mathscr{P}(\mathbb{Q}_{p}^{n})$ the set of all measurable functions $ q(\cdot): \mathbb{Q}_{p}^{n}\to(1,\infty)$ such that
$$1< q_{-}\le q(x) \le q_{+}<\infty,\ \ x\in \mathbb{Q}_{p}^{n}.$$
 \item  The set $\mathscr{B}(\mathbb{Q}_{p}^{n})$ consists of all  measurable functions  $q(\cdot)\in\mathscr{P}(\mathbb{Q}_{p}^{n})$ satisfying that the Hardy-Littlewood maximal operator $\mathcal{M}^{p}$ is bounded on $L^{q(\cdot)}(\mathbb{Q}_{p}^{n})$.
\end{enumerate}
\end{definition}

\begin{definition}[$p$-adic variable exponent Lebesgue spaces] \label{def.p-adic-lebesgue-space}
 Let   $q(\cdot) \in \mathscr{P}(\mathbb{Q}_{p}^{n})$.
Define the $p$-adic variable exponent Lebesgue spaces $L^{q(\cdot)}(\mathbb{Q}_{p}^{n})$ as follows
\begin{align*}
  L^{q(\cdot)}(\mathbb{Q}_{p}^{n})=\{f~ \text{is measurable function}:  \mathcal{F}_{q}(f/\eta)<\infty ~\text{for some constant}~ \eta>0\},
\end{align*}
where $\mathcal{F}_{q}(f):=\int_{\mathbb{Q}_{p}^{n}} |f(x)|^{q(x)} \mathrm{d}x$. The Lebesgue space $L^{q(\cdot)}(\mathbb{Q}_{p}^{n})$ is a Banach function space with respect to the Luxemburg norm
 \begin{equation*}
   \|f\|_{L^{q(\cdot)}(\mathbb{Q}_{p}^{n})}=\inf \Big\{ \eta>0:  \mathcal{F}_{q}(f/\eta)=\int_{\mathbb{Q}_{p}^{n}} \Big( \frac{|f(x)|}{\eta} \Big)^{q(x)} \mathrm{d}x \le 1 \Big\}.
\end{equation*}
\end{definition}

\begin{definition}[$\log$-H\"{o}lder continuity\cite{chacon2020variable}] \label{def.4.1-4.4-log-holder}
 Let measurable function $q(\cdot) \in \mathscr{P}(\mathbb{Q}_{p}^{n})$.
\begin{enumerate}[label=(\arabic*),itemindent=1em] 
\item Denote by  $\mathscr{C}_{0}^{\log}(\mathbb{Q}_{p}^{n})$ the set of all 
    $q(\cdotp)$ which satisfies
 \begin{equation*}    
  \gamma\Big(  q_{-}(B_{\gamma}(x)) -  q_{+}(B_{\gamma}(x)) \Big)  \le C
\end{equation*}
for all $\gamma\in \mathbb{Z}$ and any $x\in\mathbb{Q}_{p}^{n}$, where $C$ denotes a universal positive constant.
 \item  The set $\mathscr{C}_{\infty}^{\log}(\mathbb{Q}_{p}^{n})$ consists of all 
     $ q(\cdot)$ which  satisfies
\begin{equation*}
 |q(x)-q(y)| \le \frac{C}{\log_{p}(p+\min\{|x|_{p},|y|_{p}\})}
\end{equation*}
 for any $x, y\in\mathbb{Q}_{p}^{n}$, where $C$ is a  positive constant.
 \item Denote by $\mathscr{C}^{\log}(\mathbb{Q}_{p}^{n}) =\mathscr{C}_{0}^{\log}(\mathbb{Q}_{p}^{n}) \bigcap \mathscr{C}_{\infty}^{\log}(\mathbb{Q}_{p}^{n})$ the set of all global log-H\"{o}lder continuous functions $q(\cdotp)$.
\end{enumerate}
\end{definition}

In what follows, we denote   $ \mathscr{C}(\mathbb{Q}_{p}^{n}) \bigcap \mathscr{P}(\mathbb{Q}_{p}^{n})$ by  $\mathscr{P}^{\log}(\mathbb{Q}_{p}^{n})  $.
And  for a function $b$ defined on $\mathbb{Q}_{p}^{n}$, we denote
\begin{align*}
  b^{-}(x) :=- \min\{b, 0\} =
\begin{cases}
 0,  & \text{if}\ b(x) \ge 0  \\
 |b(x)|, & \text{if}\ b(x) < 0
\end{cases}
\end{align*}
and  $b^{+}(x) =|b(x)|-b^{-}(x)$. Obviously, $b(x)=b^{+}(x)-b^{-}(x)$.


\begin{definition}[$p$-adic Lipschitz space]  \label{def.lipschitz-space}
\begin{enumerate}[ label=(\arabic*)]
\item Assume that $0<\beta <1$. The p-adic version of homogeneous Lipschitz spaces $\Lambda_{\beta} (\mathbb{Q}_{p}^{n})$  is the set of all measurable functions    $f$ on $\mathbb{Q}_{p}^{n}$  with the finite norm
\begin{align*} 
  \|f\|_{\Lambda_{\beta} (\mathbb{Q}_{p}^{n})} &=  \sup_{x,y\in \mathbb{Q}_{p}^{n} \atop x\neq y} \dfrac{|f(x)-f(y)|}{|x- y|_{p}^{\beta}}.
\end{align*}
    \label{enumerate:def-lipschitz-1-he2022characterization} 
\item For $1\le q<\infty$, the p-adic version of Lipschitz spaces $\lip_{\beta}^{q} (\mathbb{Q}_{p}^{n})$ is the set of all measurable functions    $f$ on $\mathbb{Q}_{p}^{n}$ such that
\begin{align*} 
  \|f\|_{\lip_{\beta}^{q} (\mathbb{Q}_{p}^{n})} &=  \sup_{x \in \mathbb{Q}_{p}^{n} \atop \gamma\in \mathbb{Z}}   \dfrac{1}{|B_{\gamma}(x)|_{h}^{\beta/n}}  \Big(\dfrac{1}{|B_{\gamma}(x)|_{h}}  \dint_{B_{\gamma}(x)}  |f(y) - f_{B_{\gamma}(x)}|^{q} \mathd y \Big)^{1/q} <\infty,
\end{align*}
where  $f_{B_{\gamma}(x)}$ denotes the average of $f$ over $B_{\gamma}(x)$, i.e.,
$f_{B_{\gamma}(x)} = |B_{\gamma}(x)|_{h}^{-1} \dint_{B_{\gamma}(x)}   f(y)  \mathd y$.
In particular, when $q=1$, we use $\lip_{\beta}  (\mathbb{Q}_{p}^{n})$  as $\lip_{\beta}^{1} (\mathbb{Q}_{p}^{n})$.
    \label{enumerate:def-lipschitz-2}
\end{enumerate}
\end{definition}

\begin{remark}[see \cite{li2015lipschitz}]  \label{rem.lipschitz-def}
\begin{enumerate}[ label=(\roman*)]  
\item  When  $0< \beta<1$,  $\Lambda_{\beta}(\mathbb{Q}_{p}^{n})$ is just the homogeneous Besov-Lipschitz
space.
\item Since $\Lambda_{\beta}(\mathbb{R}^{n})$ and  $\bmo(\mathbb{R}^{n})$ are Campanato space when  $0< \beta<1$ and $\beta=0$,respectively. Thus, in this sense, the space $\bmo(\mathbb{Q}_{p}^{n})$  can be seen as a limit case of $\Lambda_{\beta}(\mathbb{Q}_{p}^{n})$  as $\beta\to 0$.
\end{enumerate}
\end{remark}



\subsection{Auxiliary propositions and lemmas}

The first part of \cref{lem.thm-5.2-variable-max-bounded} may be found in \cite{chacon2020variable} (see Theorem 5.2).  By elementary calculations, the second of \cref{lem.thm-5.2-variable-max-bounded} can  be obtained as well.

\begin{lemma}  \label{lem.thm-5.2-variable-max-bounded}
Let $q(\cdot)\in \mathscr{P}(\mathbb{Q}_{p}^{n} )$.
\begin{enumerate}[label=(\arabic*),itemindent=1em] 
\item  If $q(\cdot)\in \mathscr{C}^{\log}(\mathbb{Q}_{p}^{n})$,  then  $ q(\cdot)\in \mathscr{B}(\mathbb{Q}_{p}^{n})$.
 \item  The  following conditions are equivalent: 
  \begin{enumerate}[label=(\roman*),itemindent=1em,align=left]
  \item  $ q(\cdot)\in \mathscr{B}(\mathbb{Q}_{p}^{n})$,
 \item   $q'(\cdot)\in \mathscr{B}(\mathbb{Q}_{p}^{n})$,
  \item    $ q(\cdot)/q_{0}\in \mathscr{B}(\mathbb{Q}_{p}^{n})$ for some $1<q_{0}<q_{-}$,
 \item    $ (q(\cdot)/q_{0})'\in \mathscr{B}(\mathbb{Q}_{p}^{n})$ for some $1<q_{0}<q_{-}$,
\end{enumerate}
where $r'$ stand for the conjugate exponent of $r$, viz., $1=\frac{1}{r(\cdot)} + \frac{1}{r'(\cdot)}$.
\end{enumerate}
\end{lemma}

\begin{remark}   \label{rem.variable-max-bounded}
If  $ q(\cdot)\in \mathscr{B}(\mathbb{Q}_{p}^{n})$ and $s>1$, then $ s q(\cdot)\in \mathscr{B}(\mathbb{Q}_{p}^{n})$ (for the Euclidean case see Remark 2.13 of \cite{cruz2006theboundedness} for more details).
\end{remark}

We now present the following   result related to the H\"{o}lder's inequality.
The part \labelcref{enumerate:holder-p-adic}   is known as the  H\"{o}lder's inequality on Lebesgue spaces over  $p$-adic vector space $\mathbb{Q}_{p}^{n}$.  And similar to the Euclidean case, the part \labelcref{enumerate:holder-p-adic-variable} can be  deduced  by simple calculations (or see Lemma 3.8 in \cite{chacon2020variable}).  
\begin{lemma}[Generalized H\"{o}lder's inequality on  $\mathbb{Q}_{p}^{n}$] \label{lem:holder-inequality-p-adic}
 Let $\mathbb{Q}_{p}^{n}$ be an $n$-dimensional $p$-adic vector space.
\begin{enumerate}[label=(\arabic*),itemindent=1em] 
\item Suppose that  $1\le q \le\infty$ with $\frac{1}{q}+\frac{1}{q'}=1$,   and measurable functions $f\in L^{q}(\mathbb{Q}_{p}^{n})$ and $g\in L^{q'}(\mathbb{Q}_{p}^{n})$.  Then there exists a positive constant $C$ such that
\begin{align*}
   \dint_{\mathbb{Q}_{p}^{n}} |f(x)g(x)|  \mathrm{d}x \le C \|f\|_{L^{q}(\mathbb{Q}_{p}^{n})} \|g\|_{L^{q'}(\mathbb{Q}_{p}^{n})}.
\end{align*}
    \label{enumerate:holder-p-adic}
\item  Suppose that     $ q_{1}(\cdot), q_{2}(\cdot), r(\cdot) \in \mathscr{P}(\mathbb{Q}_{p}^{n})$ and   $r(\cdot) $ satisfy $\frac{1}{r(\cdot) }=\frac{1}{q_{1}(\cdot)}+ \frac{1}{q_{2}(\cdot)}$ almost everywhere.
Then there exists a positive constant $C$ such that the inequality
\begin{align*}
  \|fg\|_{L^{r(\cdot)}(\mathbb{Q}_{p}^{n})}\le C \|f \|_{L^{q_{1}(\cdot)}(\mathbb{Q}_{p}^{n})}  \|g \|_{L^{q_{2}(\cdot)}(\mathbb{Q}_{p}^{n})}
\end{align*}
holds for all  $f \in L^{q_{1}(\cdot)}(\mathbb{Q}_{p}^{n})$ and  $g \in L^{q_{2}(\cdot)}(\mathbb{Q}_{p}^{n})$.
    \label{enumerate:holder-p-adic-variable}
\item  When $r(\cdot)= 1$ in  \labelcref{enumerate:holder-p-adic-variable}  as mentioned above, we have $ q_{1}(\cdot), q_{2}(\cdot)  \in \mathscr{P}(\mathbb{Q}_{p}^{n})$ and   $\frac{1}{q_{1}(\cdot)}+ \frac{1}{q_{2}(\cdot)}=1$ almost everywhere.
Then there exists a positive constant $C$ such that the inequality
\begin{align*}
    \dint_{\mathbb{Q}_{p}^{n}} |f(x)g(x)|  \mathrm{d}x  \le C \|f \|_{L^{q_{1}(\cdot)}(\mathbb{Q}_{p}^{n})}  \|g \|_{L^{q_{2}(\cdot)}(\mathbb{Q}_{p}^{n})}
\end{align*}
holds for all  $f \in L^{q_{1}(\cdot)}(\mathbb{Q}_{p}^{n})$ and  $g \in L^{q_{2}(\cdot)}(\mathbb{Q}_{p}^{n})$.
    \label{enumerate:holder-p-adic-variable-1}
\end{enumerate}
\end{lemma}
The following results for the characteristic function $\dchi_{B_{\gamma}(x)}$ are  required as well.
By elementary calculations, the first part may be obtain  from  the $p$-adic integral theory (or refer to  \labelcref{equ:p-adic-integral-theory}).
The second part may be founded in \cite{chacon2021fractional} (see Lemma 7), and the part (4) follows from  \labelcref{enumerate:holder-p-adic-variable} of  \cref{lem:holder-inequality-p-adic}. Moreover, according to \cref{lem.thm-5.2-variable-max-bounded} and \labelcref{enumerate:charact-norm-p-adic-variable-chacon2021fractional} of \cref{lem:norm-characteristic-p-adic}, the third part  can also be deduced by simple calculations. So, we omit the proofs.
\begin{lemma}[Norms of characteristic functions]\label{lem:norm-characteristic-p-adic}
Let $\mathbb{Q}_{p}^{n}$ be an $n$-dimensional $p$-adic vector space.
 \begin{enumerate}[label=(\arabic*),itemindent=1em] 
\item If $1\le q<\infty$. Then there exist a constant $C > 0 $ such that
\begin{align*}
  \|\dchi_{B_{\gamma}(x)}\|_{L^{q}(\mathbb{Q}_{p}^{n})}    &= |B_{\gamma}(x)|_{h}^{1/q}=p^{n\gamma/q}.
\end{align*}
    \label{enumerate:charact-norm-p-adic-he}
\item  If $q(\cdot)\in  \mathscr{C}^{\log}(\mathbb{Q}_{p}^{n})$. Then
\begin{align*}
  \|\dchi_{B_{\gamma}(x)}\|_{L^{q(\cdot)}(\mathbb{Q}_{p}^{n})}    & \le C p^{n\gamma/q(x,\gamma)},
\end{align*}
where
\begin{align*}
q(x,\gamma)=
\begin{cases}
 q(x) & \text{if} \ \gamma < 0,
\\
q(\infty) & \text{if} \ \gamma \ge 0.
\end{cases}
\end{align*}
 \label{enumerate:charact-norm-p-adic-variable-chacon2021fractional}
\item  If $q(\cdot)\in  \mathscr{C}^{\log}(\mathbb{Q}_{p}^{n})$ and $q(\cdot)\in  \mathscr{P}(\mathbb{Q}_{p}^{n} )$. Then there exist a constant $C > 0 $ such that
\begin{align*}
  \dfrac{1}{|B_{\gamma} (x)|_{h} }  \big\|\dchi_{B_{\gamma} (x)} \big\|_{L^{q(\cdot)}(\mathbb{Q}_{p}^{n}) }  \big\|   \dchi_{B_{\gamma} (x)} \big\|_{L^{q'(\cdot)}(\mathbb{Q}_{p}^{n}) } <C
\end{align*}
holds for all  $p$-adic ball $B_{\gamma} (x)\subset \mathbb{Q}_{p}^{n}$.
 \label{enumerate:charact-norm-conjugate-p-adic-variable}
\item Let $0 <\alpha<n$. If  $ q(\cdot) $, $r(\cdot)\in \mathscr{P}(\mathbb{Q}_{p}^{n})$  with   $r_{+}<\frac{n}{\alpha}$ and $1/q(\cdot) = 1/r(\cdot) - \alpha/n$, then  there exists a  constant $C>0$ such that
\begin{align*}
  \|\dchi_{B_{\gamma} (x)} \|_{L^{r(\cdot)}(\mathbb{Q}_{p}^{n})}\le C  |B_{\gamma} (x)|_{h}^{\beta/n} \|\dchi_{B_{\gamma} (x)}  \|_{L^{q(\cdot)}(\mathbb{Q}_{p}^{n})}
\end{align*}
holds for all  $p$-adic balls $B_{\gamma} (x)\subset \mathbb{Q}_{p}^{n}$.
    \label{enumerate:charact-norm-fraction-p-adic-variable}
\end{enumerate}
\end{lemma}

\begin{lemma}\cite{hussain2021boundedness} \label{lem:3.1-hussain2021}
 Suppose  $f\in \Lambda_{\beta} (\mathbb{Q}_{p}^{n})$ and $0<\beta<1$, then for any $x,y\in \mathbb{Q}_{p}^{n}$, one has
\begin{align*}
  |f(x)-f(y)| &\le  |x-y|_{p}^{\beta}  \|f\|_{\Lambda_{\beta} (\mathbb{Q}_{p}^{n})}.
\end{align*}
\end{lemma}


The following are   some properties of $p$-adic  Lipschitz spaces (see \cite{he2022characterization} for more details).
\begin{lemma} \label{lem:6-he2022characterization}
   Let   $0<\beta<1$  and $1\le q<\infty$.  If $f\in \lip_{\beta}^{q} (\mathbb{Q}_{p}^{n})$.
 \begin{enumerate}[label=(\arabic*),itemindent=1em] 
\item  Then the norm $\|f\|_{\lip_{\beta}^{q} (\mathbb{Q}_{p}^{n})} $ is equivalent to the norm $\|f\|_{\lip_{\beta} (\mathbb{Q}_{p}^{n})}$.
    \label{enumerate:Lem5-property-Lip-b}

\item 
Then the homogeneous Lipschitz space $\Lambda_{\beta} (\mathbb{Q}_{p}^{n})$ coincides with the space  $\lip_{\beta}^{q} (\mathbb{Q}_{p}^{n})$.
\label{enumerate:Lem6-property-Lip}
\end{enumerate}
\end{lemma}

From the proof of Theorem 4 in \cite{he2022characterization}, we can obtain the following characterization of nonnegative Lipschitz functions.
\begin{lemma}  \label{lem:non-negative-max-lip}
 Let $0 <\beta <1$ and $b$ be a locally integrable function on $\mathbb{Q}_{p}^{n}$. Then the following assertions are equivalent:
\begin{enumerate}[label=(\arabic*),itemindent=2em]
\item   $b\in  \Lambda_{\beta}(\mathbb{Q}_{p}^{n})$  and $b\ge 0$.
     \label{enumerate:Lem-non-negative-max-lip-1}
   \item For all $1\le s<\infty$,  there exists a positive constant $C$ such that
\begin{align} \label{inequ:non-negative-max-lip}
 \sup_{\gamma\in \mathbb{Z} \atop x\in \mathbb{Q}_{p}^{n}} \dfrac{1}{|B_{\gamma} (x)|_{h}^{\beta/n}} \Big(\dfrac{1}{|B_{\gamma} (x)|_{h}} \dint_{B_{\gamma} (x)} \Big|b(y)- \mathcal{M}_{B_{\gamma} (x)}^{p} (b)(y) \Big|^{s} \mathd y \Big)^{1/s} \le C.
\end{align}
    \label{enumerate:Lem-non-negative-max-lip-2}
   \item \labelcref{inequ:non-negative-max-lip} holds for some $1\le s<\infty$.
     \label{enumerate:Lem-non-negative-max-lip-3}
\end{enumerate}
\end{lemma}

\begin{proof}

 Since the implication \labelcref{enumerate:Lem-non-negative-max-lip-2}  $\xLongrightarrow{\ \  }$ \labelcref{enumerate:Lem-non-negative-max-lip-3} follows readily, and the implication \labelcref{enumerate:Lem-non-negative-max-lip-3}  $\xLongrightarrow{\ \  }$ \labelcref{enumerate:Lem-non-negative-max-lip-1} was
proved in \cite[Theorem 4]{he2022characterization}, we only need to prove \labelcref{enumerate:Lem-non-negative-max-lip-1}  $\xLongrightarrow{\ \  }$ \labelcref{enumerate:Lem-non-negative-max-lip-2}.

If  $b\in  \Lambda_{\beta}(\mathbb{Q}_{p}^{n})$  and $b\ge 0$, then it follows from \cite[Theorem 4]{he2022characterization} that \labelcref{inequ:non-negative-max-lip} holds for all $s$ with $n/(n-\beta)<s<\infty$. Applying H\"{o}lder's inequality, we see that \labelcref{inequ:non-negative-max-lip}  holds for $1\le s\le n/(n-\beta)$ as well.

Hence, the implication \labelcref{enumerate:Lem-non-negative-max-lip-1}  $\xLongrightarrow{\ \  }$ \labelcref{enumerate:Lem-non-negative-max-lip-2} is proven.

\end{proof}

Now we recall the Hardy-Littlewood-Sobolev inequality for the  fractional maximal function $\mathcal{M}_{\alpha}^{p}$ on $p$-adic vector space.   The first part of \cref{lem:frac-max-p-adic}  can be founded in \cite{he2022characterization} (see Lemma 8), and the second comes from  Theorem 4 of \cite{chacon2021fractional}.
\begin{lemma}  \label{lem:frac-max-p-adic}
   Let $0 <\alpha<n$ and $\mathbb{Q}_{p}^{n}$ be an $n$-dimensional $p$-adic vector space.
\begin{enumerate}[label=(\arabic*),itemindent=2em]
\item  If $1<r< n/\alpha$ such that $1/q = 1/r  -\alpha/n$, then  $\mathcal{M}_{\alpha}^{p} $ is bounded from $L^{r}(\mathbb{Q}_{p}^{n})$ to $L^{q}(\mathbb{Q}_{p}^{n})$.
   \label{enumerate:lem-8-he2022characterization} 
 \item   If $r(\cdot), q(\cdot)\in   \mathscr{C}^{\log}(\mathbb{Q}_{p}^{n}) $ with $r(\cdot)\in \mathscr{P}(\mathbb{Q}_{p}^{n})$,  $r_{+}<n/\alpha$ and $1/q(\cdot) = 1/r(\cdot) -\alpha/n$, then $\mathcal{M}_{\alpha}^{p} $ is bounded from $L^{r(\cdot)}(\mathbb{Q}_{p}^{n})$ to $L^{q(\cdot)}(\mathbb{Q}_{p}^{n})$.
    \label{enumerate:thm-4-chacon2021fractional} 
\end{enumerate}
\end{lemma}

By  \labelcref{enumerate:lem-8-he2022characterization} of \cref{lem:frac-max-p-adic}, if  $0 <\alpha<n$,  $1<r< n/\alpha$ and $f\in L^{r}(\mathbb{Q}_{p}^{n})$, then $\mathcal{M}_{\alpha}^{p}(f)(x)<\infty $ almost everywhere. A similar result is also valid in variable Lebesgue spaces. And the method of proof can refer to Lemma 2.6 in \cite{zhang2019some}, so we omit its proof.
\begin{lemma}  \label{lem:frac-max-almost-every}
   Let $0 <\alpha<n$, $r(\cdot)\in \mathscr{P}(\mathbb{Q}_{p}^{n})$ and $1<r_{-}\le r_{+}<n/\alpha$. If   $f\in L^{r(\cdot)}(\mathbb{Q}_{p}^{n})$, then  $\mathcal{M}_{\alpha}^{p}(f)(x)<\infty $ for almost everywhere $x\in \mathbb{Q}_{p}^{n}$.
\end{lemma}

Now, we give the following pointwise estimate for $[b,\mathcal{M}_{\alpha}^{p}]$ when $b\in  \Lambda_{\beta}(\mathbb{Q}_{p}^{n})$.
\begin{lemma}  \label{lem:frac-max-pointwise}
   Let $0 \le \alpha<n$, $0 <\beta <1$, $0<  \alpha+\beta<n$   and $f$ be a locally integrable function on $\mathbb{Q}_{p}^{n}$.
    If $b\in  \Lambda_{\beta}(\mathbb{Q}_{p}^{n})$  and $b\ge 0$, then, for any $x\in \mathbb{Q}_{p}^{n}$ such that    $\mathcal{M}_{\alpha}^{p}(f)(x)<\infty $, we have
\begin{align*}
  \big| [b,\mathcal{M}_{\alpha}^{p}](f)(x)  \big| \le \|b\|_{\Lambda_{\beta} (\mathbb{Q}_{p}^{n})}  \mathcal{M}_{\alpha+\beta}^{p} (f)(x).
\end{align*}
\end{lemma}

\begin{proof}
  For any fixed $x\in \mathbb{Q}_{p}^{n}$ such that    $\mathcal{M}_{\alpha}^{p}(f)(x)<\infty $, if $b\in  \Lambda_{\beta}(\mathbb{Q}_{p}^{n})$  and $b\ge 0$, then
\begin{align*}
 \big| [b,\mathcal{M}_{\alpha}^{p}](f)(x)  \big| &=  \big| b(x)\mathcal{M}_{\alpha}^{p}(f)(x) -\mathcal{M}_{\alpha}^{p} (bf)(x)  \big|  \\
  &= \Big| \sup_{\gamma\in \mathbb{Z}  \atop x\in \mathbb{Q}_{p}^{n}} \dfrac{1}{|B_{\gamma} (x)|_{h}^{1-\alpha/n}}  \dint_{B_{\gamma} (x)}  b(x)  |f(y)| \mathd y - \sup_{\gamma\in \mathbb{Z}  \atop x\in \mathbb{Q}_{p}^{n}} \dfrac{1}{|B_{\gamma} (x)|_{h}^{1-\alpha/n}}  \dint_{B_{\gamma} (x)}  b(y)  |f(y)| \mathd y  \Big| \\
  &\le  \sup_{\gamma\in \mathbb{Z}  \atop x\in \mathbb{Q}_{p}^{n}} \dfrac{1}{|B_{\gamma} (x)|_{h}^{1-\alpha/n}}  \dint_{B_{\gamma} (x)} \big| b(x)-b(y) \big| |f(y)| \mathd y  \\
   &\le C \|b\|_{\Lambda_{\beta} (\mathbb{Q}_{p}^{n})} \sup_{\gamma\in \mathbb{Z}  \atop x\in \mathbb{Q}_{p}^{n}} \dfrac{1}{|B_{\gamma} (x)|_{h}^{1-\frac{\alpha+\beta}{n}}}  \dint_{B_{\gamma} (x)}   |f(y)| \mathd y   \\
     &\le  C \|b\|_{\Lambda_{\beta} (\mathbb{Q}_{p}^{n})}  \mathcal{M}_{\alpha+\beta}^{p} (f)(x).
\end{align*}

\end{proof}

Finally, we also need the following results. Similar to the proof of Lemma 2.3 in \cite{zhang2009commutators}, and referring  to  the course of the proof of Theorem 1.4 in \cite{he2023necessary},    using \cref{lem:lem-3.1-kim2009q}  and the properties of $p$-adic ball,  through elementary calculations and derivations, the following assertions can be obtained. 
Hence,  we omit the proofs.    
\begin{lemma} \label{lem:frac-max-pointwise-property} 
Let  $b$ be a locally integrable function and $\mathbb{Q}_{p}^{n}$ be an $n$-dimensional $p$-adic vector space. For any fixed  $p$-adic ball $B_{\gamma} (x)\subset \mathbb{Q}_{p}^{n}$.
 \begin{enumerate}[label=(\arabic*),itemindent=1em] 
\item  If  $0 \le \alpha<n$,     then for all $y\in B_{\gamma}(x)$, we have
\begin{align*}
   \mathcal{M}_{\alpha}^{p}(b\dchi_{B_{\gamma}(x)})(y)  &=  \mathcal{M}_{\alpha,B_{\gamma}(x)}^{p}(b)(y)
\\ \intertext{and}
    \mathcal{M}_{\alpha}^{p}(\dchi_{B_{\gamma}(x)})(y)  &=  \mathcal{M}_{\alpha,B_{\gamma}(x)}^{p}(\dchi_{B_{\gamma}(x)})(y)=|B_{\gamma}(x)|_{h}^{\alpha/n}.
\end{align*}
\item  Then for any  $y\in B_{\gamma} (x)$, we have
\begin{align*}
    |b_{B_{\gamma}(x)}|   &\le  |B_{\gamma}(x)|_{h}^{-\alpha/n}\mathcal{M}_{\alpha,B_{\gamma}(x)}^{p}(b)(y).
\end{align*}
\item  Let $E=\{y\in B_{\gamma}(x): b(y)\le b_{B_{\gamma}(x)}\}$ and $F=  B_{\gamma}(x)\setminus E =\{y\in B_{\gamma}(x): b(y)> b_{B_{\gamma}(x)}\}$. Then the following equality is trivially true
\begin{align*}
    \dint_{E} |b(y)-b_{B_{\gamma}(x)}| \mathd y  &=  \dint_{F} |b(y)-b_{B_{\gamma}(x)}| \mathd y.
\end{align*}

\end{enumerate}
\end{lemma}


\section{Proofs of the main results} 
\label{sec:result-proof}

Now we give the proofs of the \cref{thm:nonlinear-frac-max-lip} and \cref{thm:frac-max-lip}.

\subsection{Proof of \cref{thm:nonlinear-frac-max-lip}}

To prove \cref{thm:nonlinear-frac-max-lip}, we first prove the following lemma.

\begin{lemma}  \label{lem:frac-max-lip-norm}
   Let $0 < \alpha<n$, $0 <\beta <1$     and $b$ be a locally integrable function on $\mathbb{Q}_{p}^{n}$. If there exists a positive constant $C$ such that
\begin{align} \label{inequ:lem-frac-max-lip-norm}
\sup_{\gamma\in \mathbb{Z} \atop x\in \mathbb{Q}_{p}^{n}} \dfrac{1}{|B_{\gamma} (x)|_{h}^{\beta/n}} \dfrac{\Big\| \big(b -|B_{\gamma} (x)|_{h}^{-\alpha/n}\mathcal{M}_{\alpha,B_{\gamma} (x)}^{p} (b) \big) \dchi_{B_{\gamma} (x)} \Big\|_{L^{q(\cdot)}(\mathbb{Q}_{p}^{n}) }}{\|\dchi_{B_{\gamma} (x)}\|_{L^{q(\cdot)}(\mathbb{Q}_{p}^{n}) }} \le C
\end{align}
 for some  $ q(\cdot)\in   \mathscr{B}(\mathbb{Q}_{p}^{n}) $, then $b\in  \Lambda_{\beta}(\mathbb{Q}_{p}^{n})$.
\end{lemma}

\begin{proof}
 Some ideas are taken from \cite{bastero2000commutators,zhang2009commutators,zhang2014commutators} and  \cite{zhang2019some}. Reasoning as the proof of (4.4) in \cite{zhang2014commutators},
see also the proof of Lemma 3.1 in \cite{zhang2019some},   for any fixed  $p$-adic ball $B_{\gamma} (x)\subset \mathbb{Q}_{p}^{n}$, we have (see part (2) of \cref{lem:frac-max-pointwise-property})
\begin{align*}
    |b_{B_{\gamma}(x)}|   &\le  |B_{\gamma}(x)|_{h}^{-\alpha/n}\mathcal{M}_{\alpha,B_{\gamma}(x)}^{p}(b)(y), \qquad  \forall~ y\in B_{\gamma} (x).
\end{align*}
Let $E=\{y\in B_{\gamma}(x): b(y)\le b_{B_{\gamma}(x)}\}$  and $F=  B_{\gamma}(x)\setminus E =\{y\in B_{\gamma}(x): b(y)> b_{B_{\gamma}(x)}\}$, then for any $ y\in E\subset B_{\gamma} (x)$, we have $b(y)\le b_{B_{\gamma}(x)}\le  |b_{B_{\gamma}(x)}| \le  |B_{\gamma}(x)|_{h}^{-\alpha/n}\mathcal{M}_{\alpha,B_{\gamma}(x)}^{p}(b)(y)$.
It is clear that
\begin{align*}
    |b(y)-b_{B_{\gamma}(x)}|   &\le \Big| b(y) - |B_{\gamma}(x)|_{h}^{-\alpha/n}\mathcal{M}_{\alpha,B_{\gamma}(x)}^{p}(b)(y)\Big|, \qquad  \forall~ y\in E.
\end{align*}

Therefore, by using  (3) of \cref{lem:frac-max-pointwise-property},   we get
\begin{align*}
 \dfrac{1}{|B_{\gamma} (x)|_{h}^{\beta/n+1}} \dint_{B_{\gamma} (x)} \big| b(y)-b_{B_{\gamma} (x)}) \big| \mathd y &=  \dfrac{1}{|B_{\gamma} (x)|_{h}^{\beta/n+1}} \dint_{E\cup F}  \big| b(y)-b_{B_{\gamma} (x)}) \big|  \mathd y  \\
  &\le  \dfrac{2}{|B_{\gamma} (x)|_{h}^{\beta/n+1}} \dint_{E}  \Big| b(y)- |B_{\gamma}(x)|_{h}^{-\alpha/n}\mathcal{M}_{\alpha,B_{\gamma}(x)}^{p}(b)(y) \Big|  \mathd y  \\
  &\le  \dfrac{2}{|B_{\gamma} (x)|_{h}^{\beta/n+1}} \dint_{B_{\gamma} (x)}  \Big| b(y)- |B_{\gamma}(x)|_{h}^{-\alpha/n}\mathcal{M}_{\alpha,B_{\gamma}(x)}^{p}(b)(y) \Big|  \mathd y.
\end{align*}
 Using   \labelcref{enumerate:holder-p-adic-variable-1} of \cref{lem:holder-inequality-p-adic} (generalized H\"{o}lder's inequality),   \labelcref{inequ:lem-frac-max-lip-norm} and  \labelcref{enumerate:charact-norm-conjugate-p-adic-variable} of \cref{lem:norm-characteristic-p-adic}, we obtain
\begin{align*}
 \dfrac{1}{|B_{\gamma} (x)|_{h}^{\beta/n+1}} \dint_{B_{\gamma} (x)} \big| b(y)-b_{B_{\gamma} (x)}) \big| \mathd y
  &\le  \dfrac{2}{|B_{\gamma} (x)|_{h}^{\beta/n+1}} \dint_{B_{\gamma} (x)}  \Big| b(y)- |B_{\gamma}(x)|_{h}^{-\alpha/n}\mathcal{M}_{\alpha,B_{\gamma}(x)}^{p}(b)(y) \Big|  \mathd y  \\
 &\le  \dfrac{C}{|B_{\gamma} (x)|_{h}^{\beta/n+1}} \Big\| \big(b -|B_{\gamma} (x)|_{h}^{-\alpha/n}\mathcal{M}_{\alpha,B_{\gamma} (x)}^{p} (b) \big) \dchi_{B_{\gamma} (x)} \Big\|_{L^{q(\cdot)}(\mathbb{Q}_{p}^{n}) }   \\
 &\qquad \times \|\dchi_{B_{\gamma} (x)}\|_{L^{q'(\cdot)}(\mathbb{Q}_{p}^{n}) } \\
   &\le \dfrac{C}{|B_{\gamma} (x)|_{h}}  \|  \dchi_{B_{\gamma} (x)}  \|_{L^{q(\cdot)}(\mathbb{Q}_{p}^{n}) }  \|\dchi_{B_{\gamma} (x)}\|_{L^{q'(\cdot)}(\mathbb{Q}_{p}^{n}) }  \\
&\le C.
\end{align*}
So, the proof is completed by applying \cref{lem:6-he2022characterization}.

\end{proof}

\begin{figure}[!ht] \centering
 \begin{tikzpicture}[vertex/.style = {shape=circle,draw,minimum size=1em},]
  \edef\r{1.2cm}
  \tkzDefPoints{0/-2/M, 0/2/N}
  \tkzDefLine[mediator](M,N) \tkzGetPoints{x}{x'}
  \tkzInterLL(M,N)(x,x') \tkzGetPoint{O}
  \tkzInterLC[R](x',x)(O,\r) \tkzGetPoints{b}{a}
  \tkzDefRegPolygon[side,sides=5,name=P](a,b)

  \tkzSetUpPoint[size = 4,fill = black!50, color = blue]
  \tkzDrawPoints(P1,P2,P3,P4,P5)

  \tkzLabelPoint[above left,vertex](P4){$4$}
  \tkzLabelPoint[above right,vertex](P3){$3$}
  \tkzLabelPoint[below right,vertex](P2) {$2$}
  \tkzLabelPoint[below left,vertex](P1) {$1$}
    \tkzLabelPoint[above left,vertex](P5){$5$}

\tkzDrawSegment[-latex,very thick, blue!80!](P1,P2)  
\tkzDrawSegment[-latex,very thick, blue!80!](P3,P4)
\tkzDrawSegment[-latex,very thick, blue!80!](P4,P1)
     \tkzDrawSegments[-latex,dotted,very thick, black!80](P2,P3 P2,P5 P5,P4)

  \tkzLabelSegment[above,sloped,midway,font=\tiny,blue](P3,P4){$w_{34}$}
  \tkzLabelSegment[below,sloped,midway,font=\tiny,blue](P1,P2){$w_{12}$}
    \tkzLabelSegment[above right,sloped,midway,rotate=-70,shift={(4pt,10pt)},font=\tiny,blue](P1,P4){$w_{41}$}
  \tkzLabelSegment[below right,sloped,midway,rotate=-75,font=\tiny](P2,P3){$w_{23}$}
  \tkzLabelSegment[above,sloped,midway,shift={(20pt,0pt)},font=\tiny](P2,P5){$w_{25}$}
   \tkzLabelSegment[above,sloped,midway, font=\tiny](P5,P4){$w_{54}$}
\end{tikzpicture}
\vskip 6pt
\caption{Proof structure of \cref{thm:nonlinear-frac-max-lip}  \\ where $w_{ij}$ denotes $i\Longrightarrow j$}\label{fig:proof-structure-frac-max-lip-1}
\end{figure}

\begin{refproof}[Proof of  \cref{thm:nonlinear-frac-max-lip}]
Since the implications \labelcref{enumerate:thm-nonlinear-frac-max-lip-2} $\xLongrightarrow{\ \  }$ \labelcref{enumerate:thm-nonlinear-frac-max-lip-3} and \labelcref{enumerate:thm-nonlinear-frac-max-lip-5} $\xLongrightarrow{\ \ }$ \labelcref{enumerate:thm-nonlinear-frac-max-lip-4} follow readily,
and \labelcref{enumerate:thm-nonlinear-frac-max-lip-2} $\xLongrightarrow{\ \  }$ \labelcref{enumerate:thm-nonlinear-frac-max-lip-5} is similar to \labelcref{enumerate:thm-nonlinear-frac-max-lip-3} $\xLongrightarrow{\ \  }$ \labelcref{enumerate:thm-nonlinear-frac-max-lip-4},
we only need to prove \labelcref{enumerate:thm-nonlinear-frac-max-lip-1} $\xLongrightarrow{\ \  }$ \labelcref{enumerate:thm-nonlinear-frac-max-lip-2}, \labelcref{enumerate:thm-nonlinear-frac-max-lip-3} $\xLongrightarrow{\ \  }$ \labelcref{enumerate:thm-nonlinear-frac-max-lip-4} and  \labelcref{enumerate:thm-nonlinear-frac-max-lip-4} $\xLongrightarrow{\ \  }$  \labelcref{enumerate:thm-nonlinear-frac-max-lip-1}
(see \Cref{fig:proof-structure-frac-max-lip-1} for the proof structure).

 \labelcref{enumerate:thm-nonlinear-frac-max-lip-1} $\xLongrightarrow{\ \  }$ \labelcref{enumerate:thm-nonlinear-frac-max-lip-2}:
 Let $b\in  \Lambda_{\beta}(\mathbb{Q}_{p}^{n})$ and $b\ge 0$. We need to prove that $ [b,\mathcal{M}_{\alpha}^{p}] $ is bounded from $L^{r(\cdot)}(\mathbb{Q}_{p}^{n})$ to $L^{q(\cdot)}(\mathbb{Q}_{p}^{n})$ for all $r(\cdot), q(\cdot)\in   \mathscr{C}^{\log}(\mathbb{Q}_{p}^{n}) $ with $r(\cdot)\in \mathscr{P}(\mathbb{Q}_{p}^{n})$, $ r_{+}<\frac{n}{\alpha+\beta}$ and $1/q(\cdot) = 1/r(\cdot) -(\alpha+\beta)/n$. For such $r(\cdot)$ and any $f\in L^{r(\cdot)}(\mathbb{Q}_{p}^{n})$, it
follows from \cref{lem:frac-max-almost-every} that $\mathcal{M}_{\alpha}^{p}(f)(x)<\infty $ for almost everywhere $x\in \mathbb{Q}_{p}^{n}$. By  \cref{lem:frac-max-pointwise}, we have
\begin{align*}
  \big| [b,\mathcal{M}_{\alpha}^{p}](f)(x)  \big| \le \|b\|_{\Lambda_{\beta} (\mathbb{Q}_{p}^{n})}  \mathcal{M}_{\alpha+\beta}^{p} (f)(x).
\end{align*}
 Then, statement \labelcref{enumerate:thm-nonlinear-frac-max-lip-2} follows from \labelcref{enumerate:thm-4-chacon2021fractional} of  \cref{lem:frac-max-p-adic}.

\labelcref{enumerate:thm-nonlinear-frac-max-lip-3} $\xLongrightarrow{\ \  }$ \labelcref{enumerate:thm-nonlinear-frac-max-lip-4}:
For any fixed  $p$-adic ball $B_{\gamma} (x)\subset \mathbb{Q}_{p}^{n}$ and any $y\in B_{\gamma} (x)$, it follows from (1) of \cref{lem:frac-max-pointwise-property} that
\begin{align*}
   \mathcal{M}_{\alpha}^{p}(b\dchi_{B_{\gamma}(x)})(y)   =  \mathcal{M}_{\alpha,B_{\gamma}(x)}^{p}(b)(y)
 \ \text{and} \
    \mathcal{M}_{\alpha}^{p}(\dchi_{B_{\gamma}(x)})(y)   =  \mathcal{M}_{\alpha,B_{\gamma}(x)}^{p}(\dchi_{B_{\gamma}(x)})(y)=|B_{\gamma}(x)|_{h}^{\alpha/n}.
\end{align*}
Then, for  any $y\in B_{\gamma} (x)$, we have
\begin{align*}
    b(y)-|B_{\gamma} (x)|_{h}^{-\alpha/n} \mathcal{M}_{\alpha,B_{\gamma} (x)}^{p} (b)(y)   &= |B_{\gamma} (x)|_{h}^{-\alpha/n} \big( b(y) |B_{\gamma} (x)|_{h}^{\alpha/n} -\mathcal{M}_{\alpha,B_{\gamma} (x)}^{p} (b)(y) \big)    \\
   &=    |B_{\gamma} (x)|_{h}^{-\alpha/n} \big( b(y)  \mathcal{M}_{\alpha}^{p}(\dchi_{B_{\gamma}(x)})(y) -\mathcal{M}_{\alpha}^{p}(b\dchi_{B_{\gamma}(x)})(y) \big)     \\
    &=    |B_{\gamma} (x)|_{h}^{-\alpha/n}  [b,\mathcal{M}_{\alpha}^{p}] (\dchi_{B_{\gamma}(x)})(y).
\end{align*}
 Thus, for any $y\in \mathbb{Q}_{p}^{n}$, we get
\begin{align*}
   \big( b(y)-|B_{\gamma} (x)|_{h}^{-\alpha/n} \mathcal{M}_{\alpha,B_{\gamma} (x)}^{p} (b)(y) \big) \dchi_{B_{\gamma}(x)}(y) &=     |B_{\gamma} (x)|_{h}^{-\alpha/n}  [b,\mathcal{M}_{\alpha}^{p}] (\dchi_{B_{\gamma}(x)})(y) \dchi_{B_{\gamma}(x)}(y).
\end{align*}
 By using assertion  \labelcref{enumerate:thm-nonlinear-frac-max-lip-3}  and   \labelcref{enumerate:charact-norm-fraction-p-adic-variable} of \cref{lem:norm-characteristic-p-adic} (norms of characteristic functions), we have
\begin{align*}
 \Big\| \big(b -|B_{\gamma} (x)|_{h}^{-\alpha/n}\mathcal{M}_{\alpha,B_{\gamma} (x)}^{p} (b) \big) \dchi_{B_{\gamma} (x)} \Big\|_{L^{q(\cdot)}(\mathbb{Q}_{p}^{n}) }
 &\le  |B_{\gamma} (x)|_{h}^{-\alpha/n} \big\| [b,\mathcal{M}_{\alpha}^{p}] (\dchi_{B_{\gamma}(x)}) \big\|_{L^{q(\cdot)}(\mathbb{Q}_{p}^{n}) }  \\
&\le C |B_{\gamma} (x)|_{h}^{-\alpha/n} \big\|  \dchi_{B_{\gamma}(x)} \big\|_{L^{r(\cdot)}(\mathbb{Q}_{p}^{n}) }  \\
 &\le C   |B_{\gamma} (x)|_{h}^{ \beta/n}  \big\|  \dchi_{B_{\gamma}(x)} \big\|_{L^{q(\cdot)}(\mathbb{Q}_{p}^{n}) },
\end{align*}
which gives \labelcref{inequ:thm-nonlinear-frac-max-lip-4} since   $B_{\gamma} (x)$ is arbitrary and $C$ is independent of $B_{\gamma} (x)$.

\labelcref{enumerate:thm-nonlinear-frac-max-lip-4} $\xLongrightarrow{\ \  }$  \labelcref{enumerate:thm-nonlinear-frac-max-lip-1}:
By  \cref{lem:non-negative-max-lip}, it suffices to prove
\begin{align} \label{inequ:proof-lem-non-negative-max-lip-41}
 \sup_{\gamma\in \mathbb{Z} \atop x\in \mathbb{Q}_{p}^{n}} \dfrac{1}{|B_{\gamma} (x)|_{h}^{1+\beta/n}} \dint_{B_{\gamma} (x)} \Big|b(y)- \mathcal{M}_{B_{\gamma} (x)}^{p} (b)(y) \Big|  \mathd y <\infty.
\end{align}

  For any fixed  $p$-adic ball $B_{\gamma} (x)\subset \mathbb{Q}_{p}^{n}$, we have
\begin{align} \label{inequ:proof-lem-non-negative-max-lip-41-1}
\begin{split}
  & \dfrac{1}{|B_{\gamma} (x)|_{h}^{1+\beta/n}} \dint_{B_{\gamma} (x)} \Big|b(y)- \mathcal{M}_{B_{\gamma} (x)}^{p} (b)(y) \Big|  \mathd y  \\
    &\le  \dfrac{1}{|B_{\gamma} (x)|_{h}^{1+\beta/n}} \dint_{B_{\gamma} (x)} \Big|b(y)- |B_{\gamma} (x)|_{h}^{-\alpha/n}\mathcal{M}_{\alpha,B_{\gamma} (x)}^{p} (b) (y)  \Big|  \mathd y     \\
    &\qquad +  \dfrac{1}{|B_{\gamma} (x)|_{h}^{1+\beta/n}} \dint_{B_{\gamma} (x)} \Big|  |B_{\gamma} (x)|_{h}^{-\alpha/n}\mathcal{M}_{\alpha,B_{\gamma} (x)}^{p} (b) (y)-\mathcal{M}_{B_{\gamma} (x)}^{p} (b)(y) \Big|  \mathd y    \\
     &:=  I_{1}+I_{2}.
\end{split}
\end{align}
For $ I_{1}$, by applying statement \labelcref{enumerate:thm-nonlinear-frac-max-lip-4},   \labelcref{enumerate:holder-p-adic-variable-1} of \cref{lem:holder-inequality-p-adic} (generalized H\"{o}lder's inequality) and   \labelcref{enumerate:charact-norm-conjugate-p-adic-variable} of \cref{lem:norm-characteristic-p-adic}, we get
\begin{align*}
 I_{1} &=   \dfrac{1}{|B_{\gamma} (x)|_{h}^{1+\beta/n}} \dint_{B_{\gamma} (x)} \Big|b(y)- |B_{\gamma} (x)|_{h}^{-\alpha/n}\mathcal{M}_{\alpha,B_{\gamma} (x)}^{p} (b) (y)  \Big|  \mathd y     \\
 &\le  \dfrac{C}{|B_{\gamma} (x)|_{h}^{\beta/n+1}} \Big\| \big(b -|B_{\gamma} (x)|_{h}^{-\alpha/n}\mathcal{M}_{\alpha,B_{\gamma} (x)}^{p} (b) \big) \dchi_{B_{\gamma} (x)} \Big\|_{L^{q(\cdot)}(\mathbb{Q}_{p}^{n}) }  \|\dchi_{B_{\gamma} (x)}\|_{L^{q'(\cdot)}(\mathbb{Q}_{p}^{n}) } \\
  &\le  \dfrac{C}{|B_{\gamma} (x)|_{h}^{\beta/n}}\dfrac{ \Big\| \big(b -|B_{\gamma} (x)|_{h}^{-\alpha/n}\mathcal{M}_{\alpha,B_{\gamma} (x)}^{p} (b) \big) \dchi_{B_{\gamma} (x)} \Big\|_{L^{q(\cdot)}(\mathbb{Q}_{p}^{n}) }}{\|\dchi_{B_{\gamma} (x)}\|_{L^{q(\cdot)}(\mathbb{Q}_{p}^{n}) }}   \\
&\le C,
\end{align*}
where the constant $C$ is independent of $B_{\gamma} (x)$.

 Now we consider $ I_{2}$. For all  $y\in B_{\gamma} (x)$,  it follows from \cref{lem:frac-max-pointwise-property}  that
\begin{align*}
   \mathcal{M}_{\alpha}^{p}(\dchi_{B_{\gamma}(x)})(y)   =  |B_{\gamma}(x)|_{h}^{\alpha/n} \ \text{and}\  \mathcal{M}_{\alpha}^{p}(b\dchi_{B_{\gamma}(x)})(y)   =  \mathcal{M}_{\alpha,B_{\gamma}(x)}^{p}(b)(y),
\\ \intertext{and}
    \mathcal{M}^{p}(\dchi_{B_{\gamma}(x)})(y)   = \dchi_{B_{\gamma}(x)}(y)   =  1 \ \text{and}\  \mathcal{M}^{p}(b\dchi_{B_{\gamma}(x)})(y)   =  \mathcal{M}_{B_{\gamma}(x)}^{p}(b)(y).
\end{align*}
Then, for any $y\in B_{\gamma} (x)$,  we get
\begin{align} \label{inequ:proof-lem-non-negative-max-lip-41-2}
\begin{split}
  &\Big|  |B_{\gamma} (x)|_{h}^{-\alpha/n}\mathcal{M}_{\alpha,B_{\gamma} (x)}^{p} (b) (y)-\mathcal{M}_{B_{\gamma} (x)}^{p} (b)(y) \Big|  \\
    &\le  |B_{\gamma} (x)|_{h}^{-\alpha/n}  \Big|  \mathcal{M}_{\alpha,B_{\gamma} (x)}^{p} (b) (y)- |B_{\gamma} (x)|_{h}^{\alpha/n} |b(y)|  \Big|  + \Big|  |b(y)|-\mathcal{M}_{B_{\gamma} (x)}^{p} (b)(y) \Big|    \\
    &\le  |B_{\gamma} (x)|_{h}^{-\alpha/n}  \Big|   \mathcal{M}_{\alpha}^{p}(b\dchi_{B_{\gamma}(x)})(y)-  |b(y)| \mathcal{M}_{\alpha}^{p}(\dchi_{B_{\gamma}(x)})(y) \Big|  \\
    &\qquad + \Big|  |b(y)|\mathcal{M}^{p}(\dchi_{B_{\gamma}(x)})(y) -\mathcal{M}^{p}(b\dchi_{B_{\gamma}(x)})(y)  \Big|    \\
  &\le  |B_{\gamma} (x)|_{h}^{-\alpha/n}  \Big|[|b|, \mathcal{M}_{\alpha}^{p}](\dchi_{B_{\gamma}(x)})(y) \Big|  + \Big| [ |b|, \mathcal{M}^{p}] (\dchi_{B_{\gamma}(x)})(y)   \Big|.
\end{split}
\end{align}

Since  $ q(\cdot)\in   \mathscr{B}(\mathbb{Q}_{p}^{n}) $ follows at once from  \cref{lem.thm-5.2-variable-max-bounded} and statement \labelcref{enumerate:thm-nonlinear-frac-max-lip-4}. Then statement \labelcref{enumerate:thm-nonlinear-frac-max-lip-4} along with \cref{lem:frac-max-lip-norm} gives  $b\in  \Lambda_{\beta}(\mathbb{Q}_{p}^{n})$, which implies $|b|\in  \Lambda_{\beta}(\mathbb{Q}_{p}^{n})$.
Thus, we can apply \cref{lem:frac-max-pointwise} to $[|b|, \mathcal{M}_{\alpha}^{p}]$ and $[ |b|, \mathcal{M}^{p}]$ due to  $|b|\in  \Lambda_{\beta}(\mathbb{Q}_{p}^{n})$ and $|b|\ge0$.

By using  \cref{lem:frac-max-pointwise} and \cref{lem:frac-max-pointwise-property} (1), for any $y\in B_{\gamma} (x)$,  we have
\begin{align*}
  \Big|[|b|, \mathcal{M}_{\alpha}^{p}](\dchi_{B_{\gamma}(x)})(y) \Big| &\le \|b\|_{\Lambda_{\beta} (\mathbb{Q}_{p}^{n})}  \mathcal{M}_{\alpha+\beta}^{p} (\dchi_{B_{\gamma}(x)})(y) \le C \|b\|_{\Lambda_{\beta} (\mathbb{Q}_{p}^{n})} |B_{\gamma} (x)|_{h}^{(\alpha+\beta)/n}
\\ \intertext{and}
 \Big| [ |b|, \mathcal{M}^{p}] (\dchi_{B_{\gamma}(x)})(y)   \Big| &\le \|b\|_{\Lambda_{\beta} (\mathbb{Q}_{p}^{n})}  \mathcal{M}_{\beta}^{p} (\dchi_{B_{\gamma}(x)})(y) \le C \|b\|_{\Lambda_{\beta} (\mathbb{Q}_{p}^{n})} |B_{\gamma} (x)|_{h}^{\beta/n}.
\end{align*}

Hence, it follows from  \labelcref{inequ:proof-lem-non-negative-max-lip-41-2} that
\begin{align*}
 I_{2} &=   \dfrac{1}{|B_{\gamma} (x)|_{h}^{1+\beta/n}} \dint_{B_{\gamma} (x)} \Big|  |B_{\gamma} (x)|_{h}^{-\alpha/n}\mathcal{M}_{\alpha,B_{\gamma} (x)}^{p} (b) (y)-\mathcal{M}_{B_{\gamma} (x)}^{p} (b)(y) \Big|  \mathd y    \\
 &\le \dfrac{C}{|B_{\gamma} (x)|_{h}^{1+(\alpha+\beta)/n}} \dint_{B_{\gamma} (x)}  \Big|[|b|, \mathcal{M}_{\alpha}^{p}](\dchi_{B_{\gamma}(x)})(y) \Big| \mathd y   +\dfrac{C}{|B_{\gamma} (x)|_{h}^{1+\beta/n}} \dint_{B_{\gamma} (x)} \Big|  [ |b|, \mathcal{M}^{p}] (\dchi_{B_{\gamma}(x)})(y)   \Big|  \mathd y    \\
     &\le C \|b\|_{\Lambda_{\beta} (\mathbb{Q}_{p}^{n})}.
\end{align*}

Putting the above estimates for $I_{1}$ and $I_{2}$ into \labelcref{inequ:proof-lem-non-negative-max-lip-41-1}, we obtain \labelcref{inequ:proof-lem-non-negative-max-lip-41}.


 This completes the proof of \cref{thm:nonlinear-frac-max-lip}.
\end{refproof}

\subsection{Proof  of \cref{thm:frac-max-lip} }

\begin{refproof}[Proof of \cref{thm:frac-max-lip}]
Similar to prove \cref{thm:nonlinear-frac-max-lip},
 we only need to prove the implications \labelcref{enumerate:thm-frac-max-lip-1} $\xLongrightarrow{\ \  }$ \labelcref{enumerate:thm-frac-max-lip-2}, \labelcref{enumerate:thm-frac-max-lip-3} $\xLongrightarrow{\ \  }$ \labelcref{enumerate:thm-frac-max-lip-4}
and \labelcref{enumerate:thm-frac-max-lip-4} $\xLongrightarrow{\ \  }$  \labelcref{enumerate:thm-frac-max-lip-1} (the proof structure is also shown in \Cref{fig:proof-structure-frac-max-lip-1}).

 \labelcref{enumerate:thm-frac-max-lip-1} $\xLongrightarrow{\ \  }$ \labelcref{enumerate:thm-frac-max-lip-2}:
 If $b\in  \Lambda_{\beta}(\mathbb{Q}_{p}^{n})$, then
for any $p$-adic ball $B_{\gamma}(x) \subset \mathbb{Q}_{p}^{n}$, using  \cref{lem:3.1-hussain2021},  we derive
\begin{align*}
\begin{split}
 \mathcal{M}_{\alpha,b}^{p} (f)(x) &= \sup_{\gamma\in \mathbb{Z}  \atop x\in \mathbb{Q}_{p}^{n}} \dfrac{1}{|B_{\gamma} (x)|_{h}^{1-\alpha/n}}  \dint_{B_{\gamma} (x)} |b(x)-b(y)| |f(y)| \mathd y  \\
  &\le  \|b\|_{\Lambda_{\beta} (\mathbb{Q}_{p}^{n})} \sup_{\gamma\in \mathbb{Z}  \atop x\in \mathbb{Q}_{p}^{n}} \dfrac{1}{|B_{\gamma} (x)|_{h}^{1-\alpha/n}}  \dint_{B_{\gamma} (x)} |x-y|_{p}^{\beta} |f(y)| \mathd y   \\
   &\le C \|b\|_{\Lambda_{\beta} (\mathbb{Q}_{p}^{n})} \sup_{\gamma\in \mathbb{Z}  \atop x\in \mathbb{Q}_{p}^{n}} \dfrac{1}{|B_{\gamma} (x)|_{h}^{1-\frac{\alpha+\beta}{n}}}  \dint_{B_{\gamma} (x)}   |f(y)| \mathd y   \\
     &\le  C \|b\|_{\Lambda_{\beta} (\mathbb{Q}_{p}^{n})}  \mathcal{M}_{\alpha+\beta}^{p} (f)(x).
\end{split}
\end{align*}

This, together with \labelcref{enumerate:thm-4-chacon2021fractional}  of \cref{lem:frac-max-p-adic}, shows that $ \mathcal{M}_{\alpha,b}^{p}$ is bounded from $L^{r(\cdot)}(\mathbb{Q}_{p}^{n})$ to $L^{q(\cdot)}(\mathbb{Q}_{p}^{n})$.

 \labelcref{enumerate:thm-frac-max-lip-3} $\xLongrightarrow{\ \  }$ \labelcref{enumerate:thm-frac-max-lip-4}: For any fixed  $p$-adic ball $B_{\gamma} (x)\subset \mathbb{Q}_{p}^{n}$. By using \cref{lem:lem-3.1-kim2009q}, for all $y\in B_{\gamma} (x)$, we have
\begin{align*}
  |b(y)-b_{B_{\gamma}(x)}|  &\le    \dfrac{1}{|B_{\gamma} (x)|_{h}} \dint_{B_{\gamma} (x)} \big|b(y)-b(z) \big| \mathd z    \\
   &=    \dfrac{1}{|B_{\gamma} (x)|_{h}} \dint_{B_{\gamma} (x)} \big|b(y)-b(z) \big| \dchi_{B_{\gamma} (x)}(z) \mathd z    \\
   &\le  \dfrac{1}{|B_{\gamma} (x)|_{h}^{\alpha/n}}   \mathcal{M}_{\alpha,b}^{p} (\dchi_{B_{\gamma} (x)})(y).
\end{align*}
 Then, for all $y\in \mathbb{Q}_{p}^{n}$, we get
\begin{align*}
  \big| \big(b(y)-b_{B_{\gamma}(x)} \big)\dchi_{B_{\gamma} (x)}(y) \big|  &\le  |B_{\gamma} (x)|_{h}^{-\alpha/n}    \mathcal{M}_{\alpha,b}^{p} (\dchi_{B_{\gamma} (x)})(y).
\end{align*}
Since  $ \mathcal{M}_{\alpha,b}^{p}$ is bounded from $L^{r(\cdot)}(\mathbb{Q}_{p}^{n})$ to $L^{q(\cdot)}(\mathbb{Q}_{p}^{n})$, using  \labelcref{enumerate:charact-norm-fraction-p-adic-variable} of \cref{lem:norm-characteristic-p-adic}, we have
\begin{align*}
  \big\| \big(b-b_{B_{\gamma}(x)} \big)\dchi_{B_{\gamma} (x)}  \big\|_{L^{q(\cdot)}(\mathbb{Q}_{p}^{n}) }   &\le  |B_{\gamma} (x)|_{h}^{-\alpha/n}    \big\| \mathcal{M}_{\alpha,b}^{p} (\dchi_{B_{\gamma} (x)}) \big\|_{L^{q(\cdot)}(\mathbb{Q}_{p}^{n}) }     \\
   &\le C |B_{\gamma} (x)|_{h}^{-\alpha/n}    \big\| \dchi_{B_{\gamma} (x)}  \big\|_{L^{r(\cdot)}(\mathbb{Q}_{p}^{n}) }     \\
   &\le C    |B_{\gamma} (x)|_{h}^{\beta/n} \|\dchi_{B_{\gamma} (x)}  \|_{L^{q(\cdot)}(\mathbb{Q}_{p}^{n})},
\end{align*}
which implies \labelcref{inequ:thm-frac-max-lip-4} since $B_{\gamma} (x)$ is arbitrary and   $C$ is independent of $B_{\gamma} (x)$.

\labelcref{enumerate:thm-frac-max-lip-4} $\xLongrightarrow{\ \  }$  \labelcref{enumerate:thm-frac-max-lip-1}:
  For any   $p$-adic ball $B_{\gamma} (x)\subset \mathbb{Q}_{p}^{n}$,  by using \labelcref{enumerate:holder-p-adic-variable-1} of \cref{lem:holder-inequality-p-adic} (generalized H\"{o}lder's inequality),  assertion  \labelcref{enumerate:thm-frac-max-lip-4} and  \labelcref{enumerate:charact-norm-conjugate-p-adic-variable} of  \cref{lem:norm-characteristic-p-adic}, we obtain
\begin{align*}
 \dfrac{1}{|B_{\gamma} (x)|_{h}^{1+\beta/n}} \dint_{B_{\gamma} (x)} \big| b(y)-b_{B_{\gamma} (x)}) \big| \mathd y &=   \dfrac{1}{|B_{\gamma} (x)|_{h}^{1+\beta/n}} \dint_{B_{\gamma} (x)} \big| b(y)-b_{B_{\gamma} (x)}) \big| \dchi_{B_{\gamma}(x)}(y) \mathd y \\
  &\le \dfrac{C}{|B_{\gamma} (x)|_{h}^{1+\beta/n}} \big\| \big(b-b_{B_{\gamma}(x)} \big)\dchi_{B_{\gamma} (x)}  \big\|_{L^{q(\cdot)}(\mathbb{Q}_{p}^{n}) }   \|\dchi_{B_{\gamma} (x)}  \|_{L^{q'(\cdot)}(\mathbb{Q}_{p}^{n})}   \\
   &=  \dfrac{C}{|B_{\gamma} (x)|_{h}^{\beta/n}} \dfrac{\big\| \big(b-b_{B_{\gamma}(x)} \big)\dchi_{B_{\gamma} (x)}  \big\|_{L^{q(\cdot)}(\mathbb{Q}_{p}^{n}) }}{\|\dchi_{B_{\gamma} (x)}  \|_{L^{q(\cdot)}(\mathbb{Q}_{p}^{n})} }       \\
  &\qquad \times    \dfrac{1}{|B_{\gamma} (x)|_{h}}  \|\dchi_{B_{\gamma} (x)}  \|_{L^{q(\cdot)}(\mathbb{Q}_{p}^{n})} \|\dchi_{B_{\gamma} (x)}  \|_{L^{q'(\cdot)}(\mathbb{Q}_{p}^{n})}    \\
&\le  \dfrac{C}{|B_{\gamma} (x)|_{h}^{\beta/n}} \dfrac{\big\| \big(b-b_{B_{\gamma}(x)} \big)\dchi_{B_{\gamma} (x)}  \big\|_{L^{q(\cdot)}(\mathbb{Q}_{p}^{n}) }}{\|\dchi_{B_{\gamma} (x)}  \|_{L^{q(\cdot)}(\mathbb{Q}_{p}^{n})} }       \\
&\le C.
\end{align*}

 This shows that $b\in  \Lambda_{\beta}(\mathbb{Q}_{p}^{n})$ by \cref{lem:6-he2022characterization} and  \cref{def.lipschitz-space} since the constant $C$ is independent of $B_{\gamma} (x)$.

 The proof of \cref{thm:frac-max-lip} is finished.
\end{refproof}


 \subsubsection*{Funding information:}
 This work was partly supported by Project of Heilongjiang Province Science and Technology Program(No.2019-KYYWF-0909,1355ZD010), the National Natural Science Foundation of China(No.11571160), the Reform and Development Foundation for Local Colleges and Universities of the Central Government(No.2020YQ07) and the Scientific Research Fund of MNU (No.D211220637).

 \subsubsection*{Data availability statement:}   
 This manuscript has no associate data.

%
%
%


\phantomsection
\addcontentsline{toc}{section}{References}
\bibliographystyle{tugboat}          

\bibliography{wu-reference}

\end{document}